\documentclass{article}
\usepackage{amsmath,amssymb,amsthm,graphics,color}
\newtheorem{thm}{Theorem}[subsection]
\newtheorem{cor}[thm]{Corollary}
\newtheorem{lem}[thm]{Lemma}
\newtheorem{prop}[thm]{Proposition}

\newtheorem{rem}[thm]{Remark}

\usepackage{color}

\newcommand{\EE}{\mathbb{E}}
\newcommand{\NN}{\mathbb{N}}
\newcommand{\PP}{\mathbb{P}}
\newcommand{\RR}{\mathbb{R}}

\newcommand{\Cc}{\mathcal{C}}

\newcommand{\Ff}{\mathcal{F}}

\newcommand{\ind}{\mathbb{I}}

\newcommand{\law}{\;\mathop{=}\limits^{(d)}\;}
\newcommand{\equi}{\mathop{\sim}\limits}

\usepackage{tikz}

\title{A stable Langevin model with diffusive-reflective boundary conditions}
\author{J.-F. Jabir\footnote{National Research University Higher School of Economics, Moscow, Russian Federation; E-mail: jjabir@hse.ru.}\; and C. Profeta\footnote{LaMME, Univ Evry, CNRS, Universit\'e Paris-Saclay, 91025, Evry, France; christophe.profeta@univ-evry.fr.} }

\date{\empty}
\begin{document}
\maketitle
\begin{abstract}In this note, we consider the construction of a one-dimensional stable Langevin type process confined in the upper half-plane and submitted to reflective-diffusive boundary conditions whenever the particle position hits $0$. We show that two main different regimes appear according to the values of the chosen parameters. We then use this study to construct the law of a (free) stable Langevin process conditioned to stay positive, thus extending earlier works on integrated Brownian motion. This construction further allows to obtain the exact asymptotics of the persistence probability of the integrated stable L\'evy process. In addition, the paper is concluded by solving the associated  trace problem in the symmetric case.
\end{abstract}


\section{Introduction}\label{sec:1}
\noindent
Let $(L_t,\,t\geq 0)$ be a strictly $\alpha$-stable L\'evy process, defined on some filtered probability space $(\Omega,\Ff,(\Ff_t,\,t\geq 0),\PP)$, with scaling parameter $\alpha\in(0,2]$ and positivity parameter $\rho$. Its characteristic exponent is chosen as (see Zolotarev \cite[p.17]{Zolo-86}):
$$\Psi(\lambda)\; =\;\log(\EE[e^{i \lambda L_1}])\; =\; -(i \lambda)^\alpha e^{-i\pi\alpha\rho\, \text{sgn}(\lambda)}, \qquad \lambda\in \RR,$$
where the positivity parameter $\rho$ is given via the usual asymmetric parameter $\beta\in[-1,1]$ by
$$
\rho=\PP(L_1\geq 0)=\frac{1}{2}+\frac{1}{\pi\alpha}\arctan(\beta\tan(\pi\alpha/2)).
$$
We assume that $|L|$ is not a subordinator, i.e. $\rho\in (0,1)$. For $\theta, c$  two (strictly) positive constants,  we consider the SDE:
\begin{equation}
\label{eq:ConfinedModel}
\left\{
\begin{aligned}
&X_t=X_0+\int_0^t U_s\,ds,\\
&U_t=U_0+L_t+\sum_{n\geq 1}\left((1-\beta_n)(\theta^n M_n-U_{\tau_n^-})-\beta_n (1+ c)U_{\tau_n^-}\right)\ind_{\{\tau_n\leq t\}},
\end{aligned}
\right.
\end{equation}
where $(\tau_n,\,n\in \mathbb{N})$ are the successive hitting times of $(X_t,\,t\geq 0)$ at the boundary $x=0$; namely
\[
\tau_n=\inf\{t> \tau_{n-1};\,X_t=0\},\,\qquad \tau_0=0,
\]
with
\begin{equation*}
\ind_{\{\tau_n\leq t\}}=\left\{
\begin{aligned}
&1\;\mbox{if}\,\tau_n\leq t,\\
&0\;\mbox{otherwise},
\end{aligned}
\right.
\end{equation*}
and where the sequences $(\beta_n,\,n\geq 1)$ and $(M_n,\,n\geq 1)$  are independent random variables, also independent from $(X_0, U_0)$ and $(L_t,\,t\geq 0)$, such that
\begin{enumerate}
\item  the random variables $(\beta_n,\,n\geq 1)$ are  i.i.d.  Bernoulli r.v.'s with parameter $p:=\PP(\beta_1=1)$,
\item the random variables $(M_n,\,n\geq 1)$ are i.i.d., non-negative and  such that $\PP(M_1=0)=0$. We further assume that they admit moment of order at least $\alpha$.
\end{enumerate}
The model \eqref{eq:ConfinedModel} describes a type of Langevin model where, at each time $t$, the components $X_t$ and $U_t$ represent respectively the position and the velocity of some fluid particle, interacting with a physical wall located at the axis $x=0$. The particle positions are confined within the upper-half space $[0,\infty)$, the velocity paths are c\`adl\`ag, governed
by the L\'evy process $(L_t,\,t\geq 0)$ when $X_t$ is in the interior $(0,\infty)$ and, whenever the particle hits the frontier $x=0$, the velocity component is
submitted to either a partially absorbing boundary condition or a diffusive boundary condition, respectively quantified by $c $ and $\theta^nM_n$. More precisely, for all $n$, the left-hand limit $U_{\tau^-_n}=\lim\limits_{t\rightarrow \tau_n,t< \tau_n}U_t$ describes the ingoing velocity of the particle and
 the reflective-diffusive interaction between the particle and the confinement frontier implies that the velocity $U_{\tau_n}=\lim\limits_{t\rightarrow \tau_n,t> \tau_n}U_t$ immediately after the impact is given by :
\begin{equation}
\label{eq:boundaryVelocity}
U_{\tau_n}=\triangle U_{\tau_n}+U_{\tau_n^-}=
\left\{
\begin{aligned}
&-cU_{\tau^-_n}\,\mbox{ if}\,\beta_n=1,\\
&\\
&\theta^n M_n\,\mbox{ if}\,\beta_n=0.
\end{aligned}
\right.
\end{equation}
As $U_{\tau^-_n}$ is necessarily non-positive, the particle either re-emerges in $(0,\infty)$ or remains stuck to the axis $x=0$, and, in any cases, remains confined within $[0,\infty)$.\\

The   (partially) absorbing case ($p=1$, i.e. $\beta_n=1$ a.s.) was previously introduced and intensively investigated in Bertoin \cite{Bertoin-07,Bertoin-08} (in the case of a totally absorbing wall, $c=0$) and Jacob \cite{Jacob-12,Jacob-13} when $L=B$ is a one dimensional Brownian motion. Jacob \cite{Jacob-12} exhibited the critical level $c_{\text{crit}}=\exp(-\pi/\sqrt{3})$ separating sticky and non-sticky situations; whenever $c\geq c_{\text{crit}}$, $\lim_n\tau_n=\infty$ a.s., whereas if $c< c_{\text{crit}}$, $\lim_n\tau_n<\infty$ a.s. In the situation where $L=B$ is a standard Brownian motion, the corresponding zeroes of the integrated Brownian motion $(x+tu+\int_0^t B_s\,ds, \, t\geq0)$ have been the subject of a long history of studies starting from the early work of McKean \cite{McKean-63} and the numerous works of Aim\'e Lachal (see \cite{Lachal-97} among others). In the case of a stable L\'evy process, few results on the distribution of the zeroes of an integrated L\'evy process are known. Nevertheless the necessary and sufficient conditions ensuring the non-accumulation of $(\tau_n,\,n\geq 1)$ in finite time will be exhibited in Section \ref{sec:2}.\\

The (totally) diffusive situation ($p=0$, i.e. $\beta_n=0$ a.s.) models the particular case of Maxwell boundary conditions introduced in the kinetic theory of gases (see e.g. Chapter $8$ in Cercignani, Illner and Pulvirenti \cite{CerIllPul-94}). The particular situation where $\theta=1$ and $(M_n, n\geq1)$ is distributed according to a Maxwellian distribution of the form:
$$
\frac{v}{\Theta}\exp\left\{-\frac{|v|^{2}}{2\Theta}\right\}\ind_{\{v\geq 0\}},\qquad \text{with }\Theta>0,
$$
 corresponds to the situation where a (gas) particle interacts with a surface in thermodynamical equilibrium at temperature $\Theta$, and where the particle re-emerges from the wall with a velocity $M_n$ after each impact.

 The introduction of the term $\theta^n$ enables to balance the effects of the reflective and diffusive boundary conditions, softening (when $\theta<1$) or increasing $(\theta>1)$ the heat transfer from the wall to the particle. In particular $\theta^n$ allows to exhibit different asymptotic regimes for the sequence $(\tau_n,\,n\geq 1)$, see Theorems 2.1.1 and 2.2.1. The addition of such component provides a peculiar wall-particle interaction when compared to classical boundary condition for kinetic equations (see Section \ref{sec:4}) as it is inhomogeneous and induces a dependency of the boundary condition with respect to the past trajectory of the particle near the wall. %

 Comparable, yet distinct, situations are currently studied in the physics literature, for instance in Mohammadzadeh and Struchtrup \cite{MohStr-15} where the Maxwellian distribution depends on the ingoing velocity of the particle, or in Chibbaro and Minier \cite{ChiMin-08} where the boundary condition, modeling particle deposition, depends on the time that the particle has spent in a near-wall region. \eqref{eq:ConfinedModel} can be thought as a illustrative and prototypical model which will be used for investigating more general situations (time-velocity dependent absorption diffusion coefficient, multidimensional setting, ...) in future works.

 Let us also point out that the connection between transport equation endowing diffusive-reflective boundary conditions and Langevin models driven by a Poisson process has been studied in Costantini and Kurtz \cite{CosKur-06} and Costantini \cite{Costantini-91}. The link between \eqref{eq:ConfinedModel} and kinetic equations will be further discussed in Section \ref{sec:4}.\\

Our aim in this paper is to study the limit behaviour of the sequence $(\tau_n,\,n\geq 1)$, according to the parameters $\theta$ and $c$. Such results then allow to discuss some related problems for the construction of a conditioned stable Langevin process, as well as some relations to kinetic equations.
\begin{itemize}
\item[$a)$] In the next Section \ref{sec:2}, we describe the asymptotic behaviour of the sequence $(\tau_n,\,n\geq 1)$ in the case where $(X_0,U_0)$ belongs to the semi-finite line $\{0\}\times(0,\infty)$, showing necessary and sufficient conditions on the parameters $c$ and $\theta$ to characterize sticky situations (see Theorem \ref{theo:1}) in the case of an absorbing wall ($p=1$), of a diffusive wall ($p=0$) and mixed boundary condition ($0<p<1$). When the lifetime of the process is infinite, we further give some a.s. asymptotics for the behavior of $(\tau_n,\,n\geq 1)$ (see Theorem \ref{theo:2}).
\item[$b)$] Section \ref{sec:3} is dedicated to the construction of an  integrated stable L\'evy  process conditioned to never hit $0$, thus extending the previous results obtained in Jacob \cite{Jacob-13} and in Groeneboom, Jongbloed and Wellner \cite{GJW} for Brownian motion. As a by product, we deduce the asymptotic behavior of the upper tail distribution of $\tau_1$ (Corollary \ref{cor:persistence}), improving the result in Profeta and Simon \cite{ProSim-15}.
\item[$c)$] Section \ref{sec:4} is dedicated to the link between \eqref{eq:ConfinedModel}{, in the case $\theta=1$,} and classical trace problems for kinetic equations. Such link was previously studied in Bossy and Jabir \cite{BoJa-11} for Langevin model driven by a Brownian motion and singular nonlinear (in the sense of McKean-Vlasov) drift component. We show the existence of trace functions under appropriate assumptions on $c,\theta$ and the distribution of $(X_0,U_0)$.
\end{itemize}

\section{Estimation on the asymptotic behavior of $(\tau_n,\,n\in\NN)$ starting from $\{0\}\times (0,+\infty)$}\label{sec:2}

We start by decomposing the paths of $X$ into a sum of excursions. To simplify the expressions, we shall assume that  a.s. $X_0=0$ and $U_0>0$.\\
For $n\geq 1$, we define the restarting velocity after the $n^\text{th}$ passage time at the boundary :
\begin{equation}
\label{eq:Vnsansgn}
V_n =U_{\tau_n}= (1-\beta_n) \theta^n M_n + \beta_n c\,  |U_{\tau_n^-}|\qquad\text{and}\qquad V_0=U_0.
\end{equation}
Observe that by the scaling property of $L$, for $n\geq1$ :
\begin{equation}\label{eq:scaling}
(\tau_n, U_{\tau_n^-}) \law (\tau_{n-1} +V_{n-1}^\alpha \xi_1,  V_{n-1} \ell_1)
\end{equation}
where the pair $(\xi_1, \ell_1)$ is independent from $V_{n-1}$ and is distributed as $(\tau_1, U_{\tau_1^-})$ when the process $(X,U)$ is started from $(0,1)$. Although the law of the pair $(\xi_1, \ell_1)$ is not explicitly known, we have the following estimates from \cite{ProSim-15}: for $\lambda\geq0$,
\begin{equation}\label{eq:moments}
\EE[\xi_1^\lambda] < +\infty\quad  \Longleftrightarrow\quad \EE[|\ell_1|^{\alpha\lambda}] < +\infty\quad  \Longleftrightarrow\quad \lambda< \frac{1-\rho}{1+\alpha\rho},
\end{equation}
as well as the Mellin's transform,
 \begin{equation}\label{eq:Utauy0}
\EE[|\ell_1|^{\nu-1}]=   \EE_{(0,1)}[|U_{\tau_1^-}|^{\nu-1}]= \frac{\sin (\pi \gamma  \nu)}{\sin\left(\pi \nu (1-\gamma)\right)}
\end{equation}
where $\gamma$ is given by
\begin{equation}\label{gamma}
\gamma = \frac{\alpha(1-\rho)}{1+\alpha}.
\end{equation}
By defining
$$g_n = \sup\{k\leq n-1;\; \beta_k=0\}\qquad \text{and}\qquad M_0=U_0,$$
we deduce by iteration from (\ref{eq:Vnsansgn}) and (\ref{eq:scaling}) that
\begin{equation}\label{eq:Vn}
(V_n, \;n\geq1) \mathop{=}\limits^{(d)} \left( (1-\beta_n) \theta^n M_n + \beta_n \theta^{g_n} M_{g_n} \prod_{i=g_n+1}^n c |\ell_i|,\; n\geq1\right)
\end{equation}
where $(\ell_n, n\geq1)$ are i.i.d. random variables with the same law as $U_{\tau_1^-}$ when $(X,U)$ starts from $(0,1)$. Note that the r.v.'s $(V_n, n\geq1)$ are of course not independent, except if $p=0$.\\

\noindent
The same scaling and the independent increments properties further allow to decompose the passage time $\tau_n$ as :
\begin{equation}\label{eq:taun}
\tau_n= \tau_1+ \sum_{k=2}^{n} \frac{\tau_{k} - \tau_{k-1}}{V_{k-1}^{\alpha}}\, V_{k-1}^\alpha\;
 \mathop{=}\limits^{(d)} \; \xi_1 U_0^\alpha   +\sum_{k=2}^{n} \xi_k V_{k-1}^\alpha
\end{equation}
where $(\xi_k, k\geq1) $ are i.i.d. r.v.'s  with the same law as $\tau_1$ when $(X,U)$ is started from $(0,1)$. Furthermore, for every fixed $n\geq1$, $\xi_k$ is independent from $V_{k-1}$.\\

\noindent
In the following, we shall be interested in the study of
$$\tau_\infty = \lim_{n\rightarrow +\infty} \tau_n = \inf\{t>0;\; (X_t, U_t) = (0,0)\}.$$

\subsection{Absorption in finite time}
\noindent
We start by looking at the conditions under which the particle is absorbed at the boundary in finite time, that is  $\tau_\infty<\infty$ a.s.
 { We also investigate (still in the case $\tau_\infty<\infty$)
the existing moments of $\tau_\infty$, namely we study for which
$\lambda>0$ we have $\EE[\tau^\lambda_\infty]<\infty$. Noting that $\EE[\tau_1]=\EE[\tau_\infty]=\EE[|l_1|^\alpha]=+\infty$, only the case $\lambda\in(0,1)$ is of interest.}
\begin{thm}\label{theo:1}
Assume that $X_0=0$ and $U_0>0$ with $U_0^{\alpha}$  integrable. Then we have the following situations:
\begin{enumerate}
\item  If $p=1$, then
$$\tau_\infty < \infty \quad\PP-\text{a.s.}\quad \Longleftrightarrow \quad c<c_{\text{crit}}$$
where
$$c_{\text{crit}}=\exp\left(- \pi\cot\left(\pi\gamma\right)\right).$$

In particular, for $0<\lambda<1$,
$$\EE[\tau_\infty^\lambda]<+\infty  \quad \Longleftrightarrow \quad  \left\{c <  c_{\text{crit}}  \text{ and } c^{\alpha\lambda}  \EE\left[|\ell_1|^{\alpha\lambda}\right]<1\right\}.$$
\item If $p=0$, then
$$\tau_\infty < \infty  \quad\PP-\text{a.s.}\quad \Longleftrightarrow \quad \theta<1. $$
In particular, for $0<\lambda<1$,
$$ \EE[\tau_\infty^\lambda]<+\infty  \quad \Longleftrightarrow \quad  \left\{\theta<1 \text{ and }  \lambda < \frac{1-\rho}{1+\alpha \rho}\right\}$$
\item If $0<p<1$, then
$$\tau_\infty < \infty \quad\PP-\text{a.s.} \quad \Longleftrightarrow \quad \theta<1. $$
In particular, for $0<\lambda<1$,
$$\EE[\tau_\infty^\lambda]<+\infty  \quad \Longleftrightarrow \quad \left\{\theta<1 \text{ and } c^{\alpha\lambda}  \EE\left[|\ell_1|^{\alpha\lambda}\right] p<1\right\}.$$
\end{enumerate}
\end{thm}

\noindent
The proof of this theorem will rely on the following lemma, which is a special case of Kolmogorov's three-series theorem (see e.g. Durrett \cite[p.64]{Durret-96}).
\begin{lem}\label{lem:1}
Let $(X_n, n\geq0)$ be i.i.d. and positive random variables. We assume that there exists $\lambda\in(0,1)$ such that $0<\EE[X_1^\lambda]<+\infty$. Then, for $a\geq0$ :
$$\sum_{n=1}^{+\infty} X_n a^n  < +\infty\quad a.s.    \qquad \Longleftrightarrow \qquad a<1.$$
\end{lem}
\begin{proof}

If $a<1$,
$$\EE\left[\left(\sum_{n=1}^{+\infty} X_n a^n \right)^\lambda\right] \leq \sum_{n=1}^{+\infty} \EE\left[X_n^\lambda\right]a^{\lambda n} = %
%\leq \sum_{n=1}^{+\infty} \EE\left[X_n^\lambda\right]a^{\nu n} =
\frac{a^\lambda\EE\left[X_1^\lambda\right]}{1-a^\lambda}<+\infty$$
which implies that the series converges a.s.
If $a\geq 1$, since all the components are positive, we have
$$\EE\left[\exp\left(- \sum_{n=1}^{\infty} X_n a^n  \right)\right]\leq \EE\left[\exp\left(- \sum_{n=1}^{N} X_n a^n  \right)\right] \leq \EE\left[\exp\left(- X_1  \right)\right]^N \xrightarrow[N\rightarrow \infty]{} 0$$
which concludes the proof.

\end{proof}

\begin{rem}
Note that we cannot totally remove the assumption on the moments in the previous lemma. Indeed, consider for instance a sequence of positive i.i.d. random variables with distribution :
$$\PP(X_1\in dx) = \frac{\ln(2)}{x \ln^2(x)} \ind_{\{x\geq 2\}} dx.$$
Then, integrating by parts,
\[
\EE\left[\exp\left(-\lambda X_1 a^n  \right)\right] =  \int_2^{+\infty} e^{-\lambda a^n x}\frac{\ln(2)}{x \ln^2(x)} dx= e^{- 2 \lambda a^n } - \lambda a^n \int_2^{+\infty} e^{-\lambda a^n x} \frac{\ln(2)}{\ln(x)}dx.
\]

Taking $a<1$, we deduce by the Tauberian theorem that :
$$1-\EE\left[\exp\left(-\lambda X_1 a^n  \right)\right]  \mathop{\sim}_{n\rightarrow +\infty}  -\frac{\ln(2)}{n\ln(a)}$$

hence
$$\EE\left[\exp\left(-\lambda \sum_{n=1}^{N} X_n a^n  \right)\right]
=  \prod_{n=1}^N  \EE\left[\exp\left(-\lambda X_1 a^n  \right)\right]\xrightarrow[N\rightarrow +\infty]{} 0
$$
which proves that $\displaystyle \sum_{n=1}^{+\infty} X_n a^n =+\infty \; a.s.$ for any $a>0$ in this case.

\end{rem}
\noindent
\begin{proof} We shall now prove Theorem \ref{theo:1} and deal with the three cases separately.\\

\noindent
1. Assume first that $p=1$. We follow the ideas of Jacob \cite[Corollary 1]{Jacob-12}. According to \eqref{eq:Vn}, we have the decomposition $$V_n\; \mathop{=}\limits^{(d)} \; U_0 \,\prod_{i=1}^n  c |\ell_i|.$$
Therefore, from (\ref{eq:taun}),
$$\tau_\infty  \;\mathop{=}\limits^{(d)}\;  \xi_1 U_0^{\alpha}+ U_0^\alpha \sum_{n=2}^{+\infty} \xi_n     c^{\alpha (n-1)}  \prod_{i=1}^{n-1} |\ell_i|^\alpha. $$
Next, from Profeta and Simon \cite{ProSim-16}, the law of large numbers implies that :
$$ \frac{1}{n} \sum_{i=1}^{n-1} \ln(|\ell_i|) \xrightarrow[n\rightarrow +\infty]{\text{a.s.}}\EE\left[\ln(|\ell_1|)\right]=  \pi\cot\left(\pi\gamma\right) =-\ln(c_{\text{crit}})$$
hence, for any fixed $\varepsilon>0$, we have for $n$ large enough
\begin{equation}\label{eq:lncrit}
c_{\text{crit}}^{-n+n\varepsilon} \leq \prod_{i=1}^{n-1}  |\ell_i| \leq c_{\text{crit}}^{-n-n\varepsilon},
\end{equation}
and the finiteness of $\tau_\infty$ follows from (\ref{eq:moments}) and Lemma \ref{lem:1} when $c<c_{\text{crit}}$. In the same way, Lemma \ref{lem:1} implies that $\tau_\infty=\infty$ a.s. whenever $c>c_{\text{crit}}$. When $c=c_{\text{crit}}$, observe that the random walk $Z_n = \sum_{i=1}^{n}\ln(|\ell_i|)+\ln(c_{\text{crit}})$ is recurrent, hence the restarting velocity $V_n=U_0\exp(Z_n)$ does not converge to 0, which implies that $\tau_\infty=\infty$ a.s.\\
To get the condition on the moments, we then notice that $\tau_\infty$ is solution of a renewal equation :

\begin{equation}
\tau_\infty  \;\mathop{=}\limits^{(d)}\;   \xi_1 U_0^{\alpha}+   c^\alpha |\ell_1|^\alpha \tau_\infty .
\label{eq:10bis}
\end{equation}
From Goldie \cite[Theorem 4.1]{Gol}, we deduce that there exists a constant $\kappa>0$ such that :
$$\PP(\tau_\infty>t) \equi_{t\rightarrow +\infty}   \frac{\kappa}{t^{\eta(c)}}$$
where $\eta(c)>0$ is such that  $c^{\alpha\eta(c)}  \EE\left[|\ell_1|^{\alpha\eta(c)}\right]=1$. In particular, $\EE[\tau^\lambda_\infty]<\infty$ if and only if $\lambda<\eta(c)$. Point 1. then follows from the fact that for $c<c_{\text{crit}}$, the function $\lambda \rightarrow c^{\alpha\lambda}  \EE\left[|\ell_1|^{\alpha\lambda}\right]$ is convex with a negative derivative at $0^+$ given by $\alpha(\ln(c)-\ln(c_{\text{crit}}))<0$.

\noindent
2. Assume now that $p=0$. Then, the particle will always restart afresh when hitting the zero axis, namely
$V_n= \theta^n M_n  $ for $n\geq1$. Hence
$$
\tau_\infty \law  \xi_1U_0^{\alpha} + \sum_{n=2}^{+\infty} \xi_n \theta^{\alpha (n-1)} M_{n-1}^\alpha
$$
and from Lemma \ref{lem:1}, this series converges if and only if $\theta<1$, in which case the moments of $\tau_\infty$ are finite if and only if those of $\xi_1$ are (since $U_0^\alpha$ and $M_1^\alpha$ are assumed to be integrable).\\

\noindent
3. Assume finally that $0<p<1$. Observe first that since the r.v.'s $(\beta_k, k\geq1)$ only take the values 0 or 1 and all the terms are positive, $\tau_n$ may be decomposed, after a change of indices, as :
$$\tau_n \law   \xi_1U_0^{\alpha}+  \sum_{k=1}^{n-1} \xi_{k+1} (1-\beta_{k}) (\theta^{k} M_{k})^\alpha +  \sum_{k=1}^{n-1} \xi_{k+1} \beta_{k} \left(\theta^{g_{k}} M_{g_{k}} \prod_{i=g_{k}+1}^{k}  c |\ell_i|\right)^\alpha$$
with the convention that $M_0=U_0$.
Since in this case the r.v.'s $(\beta_k, k\geq1)$ take infinitely many often the value 0, we first  observe by Lemma \ref{lem:1} that $\tau_\infty = \infty$ as soon as $\theta\geq1$.

So assuming that $\theta<1$ and taking $0<\lambda<\frac{1-\rho}{1+\alpha\rho} $, we then have :

\begin{align*}
\EE[\tau_n^\lambda]& \leq \EE[\xi_1^\lambda U_0^{\lambda \alpha}  ]+  \sum_{k=1}^{n-1} \EE\left[(1-\beta_k) \left( \xi_{k+1}\theta^{\alpha k} M_k^\alpha\right)^{\lambda}\right] +\sum_{k=1}^{n-1} \EE\left[\beta_k\left(\xi_{k+1}  \theta^{\alpha g_k} M_{g_k}^\alpha \prod_{i=g_k+1}^k c^\alpha| \ell_i|^\alpha\right)^\lambda\right]\\
&\leq\EE[\xi_1^\lambda]\left( \EE[U_0^{\lambda \alpha}  ]+(1-p) \EE[M_1^{\alpha\lambda}]\frac{\theta^{\alpha\lambda} -\theta^{\alpha\lambda n} }{1-\theta^{\alpha\lambda}}\right) +     p \EE\left[\xi_1^\lambda\right] \sum_{k=1}^{n-1}   \EE\left[\theta^{\alpha g_k\lambda} M_{g_k}^{\alpha\lambda} \prod_{i=g_k+1}^k c^{\alpha\lambda} |\ell_i|^{\alpha\lambda}\right].
\end{align*}

In the following Lemma, we compute the remaining sum ($\sum_{k=1}^{n-1}(\dots)$) under the more general assumption that $\theta\neq1$, as the result for $\theta>1$ will be needed during the proof of Theorem 2.2.1.

\begin{lem}\label{lem:sumgn}
For $n\geq 2$, $\theta\neq1$ and $0<\lambda <  \frac{1-\rho}{1+\alpha\rho}$:
\begin{multline*}
\sum_{k=1}^n \EE\left[\theta^{\alpha g_k\lambda} M_{g_k}^{\alpha\lambda} \prod_{i=g_k+1}^k c^{\alpha\lambda} |\ell_i|^{\alpha\lambda}\right] \\=
\frac{1}{p}\EE[U_0^{\alpha\lambda}] \left( c^{\alpha\lambda}\EE\left[ |\ell_1|^{\alpha\lambda}\right] p\right)^n
+\frac{1}{p} \sum_{i=1}^{n-1}  \left(c^{\alpha\lambda }\EE\left[ |\ell_1|^{\alpha\lambda}\right] p\right)^i \left(
\EE[U_0^{\alpha\lambda}]  + (1-p)  \theta^{\alpha \lambda} \EE[M_1^{\alpha\lambda}] \frac{\theta^{\alpha\lambda (n-i)} - 1 }{ \theta^{\alpha \lambda}-1}\right).\end{multline*}
In particular, for $\lambda$ small enough such that $c^{\alpha\lambda}  \EE\left[|\ell_1|^{\alpha\lambda}\right] p<1$, there exist two constants $A_\lambda$, $B_\lambda$ independent of $n$ such that
$$ \EE\left[\tau_n^\lambda\right] \leq A_\lambda \theta^{\alpha \lambda n} +B_\lambda.$$

\end{lem}
\begin{proof}
The law of $g_k$ is given by :
\begin{equation}\label{gn}
\PP(g_k=l) = (1-p) p^{k-l-1} \text{ for } l\in \{1, \ldots, k-1\}\qquad \text{ and }\qquad \PP(g_k=0) = p^{k-1}.
\end{equation}
We decompose the expectation according to $g_k$. When $g_k=0$ :
\begin{align}\label{Etaugn0}
 \EE\left[\theta^{\alpha g_k\lambda} M_{g_k}^{\alpha\lambda} \prod_{i=g_k+1}^kc^{\alpha\lambda} |\ell_i|^{\alpha\lambda}\ind_{\{g_k=0\}}\right] &= \EE[U_0^{\alpha\lambda}]c^{\alpha\lambda k}\EE\left[ |\ell_1|^{\alpha\lambda}   \right]^k  p^{k-1}.
\end{align}
 When $g_k>0$, similar computations yield :
\begin{align}
\notag \EE\left[\theta^{\alpha g_k\lambda} M_{g_k}^{\alpha\lambda} \prod_{i=g_k+1}^k c^{\alpha\lambda} |\ell_i|^{\alpha\lambda}\ind_{\{g_k>0\}}\right]&= \sum_{l=1}^{k-1} \EE\left[\theta^{\alpha l \lambda} M_{l}^{\alpha\lambda} \prod_{i=l+1}^kc^{\alpha\lambda} |\ell_i|^{\alpha\lambda}\right](1-p)p^{k-l-1}\\
\label{Etaugn} &= (1-p) \EE[M_{1}^{\alpha\lambda}]  \sum_{i=1}^{k-1}  \theta^{\alpha \lambda (k-i)} c^{\alpha\lambda i} \EE\left[ |\ell_1|^{\alpha\lambda}\right]^{i} p^{i-1}.
 \end{align}
Applying Fubini-Tonelli's theorem, we deduce that
\begin{align}
\sum_{k=2}^{n} \sum_{i=1}^{k-1}  \theta^{\alpha \lambda (k-i)}
c^{\alpha\lambda i} \EE\left[ |\ell_1|^{\alpha\lambda}\right]^{i} p^{i-1}
\label{theta1} &= \sum_{i=1}^{n-1} \theta^{- \alpha \lambda i}
 c^{\alpha\lambda i}\EE\left[ |\ell_1|^{\alpha\lambda}\right]^{i} p^{i-1}
  \sum_{k=i+1}^{n}  \theta^{k \alpha \lambda} \\
\notag &= \theta^{\alpha\lambda} \sum_{i=1}^{n-1}  c^{\alpha\lambda i} \EE\left[|\ell_1|^{\alpha\lambda}\right]^{i} p^{i-1} \frac{1 -\theta^{\alpha\lambda(n-i)} }{1-\theta^{\alpha\lambda}}
\end{align}
which concludes the proof of Lemma \ref{lem:sumgn}.
\end{proof}

\noindent
Now, letting $n\rightarrow +\infty$ in Lemma \ref{lem:sumgn} and using the fact that $\theta<1$, we deduce on the one hand that $\EE[\tau_\infty^\lambda] <+\infty$ as soon as  $\lambda$ is small enough so that $c^{\alpha\lambda}  \EE\left[|\ell_1|^{\alpha\lambda}\right] p<1$. In particular, $\tau_\infty<+\infty$ a.s. for any $c>0$. On the other hand, we have :
$$\EE[\tau_\infty^\lambda] \geq \EE\left[\left( U_0^\alpha \sum_{k=1}^{+\infty} \xi_{k+1} \beta_{k} \left( \prod_{i=1}^{k}  c |\ell_i|\right)^\alpha \ind_{\{g_k=0\}}\right)^\lambda\right].$$
Observe now that the random variable appearing on the right-hand side of the previous equation, say $\chi_\infty$, is a solution of the following renewal equation :
$$\chi_{\infty} \law \xi_2 \beta_1 U_0^{\alpha} c^\alpha|\ell_1|^\alpha \ind_{\{g_1=0\}} + c^\alpha|\ell_1|^\alpha \ind_{\{\beta_1=1\}} \chi_\infty.$$
As in Point 1., Goldie's result \cite[Theorem 4.1]{Gol} implies that there exists a constant $\kappa>0$ such that
$$\PP\left(\chi_\infty>t\right) \equi_{t\rightarrow +\infty} \frac{\kappa}{t^{\lambda_0}}$$
where $\lambda_0>0$ is the solution of
$$1=\EE\left[c^{\alpha\lambda_0}|\ell_1|^{\alpha\lambda_0} \ind^{\lambda_0}_{\{\beta_1=1\}} \right] = c^{\alpha\lambda_0}\EE\left[|\ell_1|^{\alpha\lambda_0} \right] p.$$
This implies that $\EE[\chi_\infty^\lambda]=+\infty$ for $\lambda\geq \lambda_0$, which concludes the proof.
 \end{proof}

\subsection{Asymptotics of $(\tau_n, n\geq1)$ }

We now study the rate of divergence of $\tau_n$ when $\tau_n\xrightarrow[n\rightarrow +\infty]{} +\infty$.

\begin{thm}\label{theo:2}
Assume now that $\tau_\infty=+\infty$ a.s. We have the following asymptotics :
\begin{enumerate}
\item  When $p=1$ :
\begin{enumerate}
\item If $ c>c_{\text{crit}}$ :
$$ \frac{\ln(\tau_n)}{n}\xrightarrow[n\rightarrow +\infty]{a.s.}  \alpha \left(\pi \cot\left(\pi \gamma  \right)+ \ln(c)\right).$$
\item If $c=c_{\text{crit}}$ :  for any  $0<\lambda < 2$, 
$$\frac{\ln(\tau_n)}{n^{1/\lambda}} \xrightarrow[n\rightarrow +\infty]{a.s.}0.$$
\end{enumerate}
\item When $p=0$ :
\begin{enumerate}
\item If $\theta>1$ :
$$ \frac{\ln(\tau_n)}{n}\xrightarrow[n\rightarrow +\infty]{a.s.} \alpha\ln(\theta).$$
\item If $\theta=1$ : for any $0<\lambda < \frac{1-\rho}{1+\alpha\rho}$,
$$\frac{\tau_n}{n^{1/\lambda}} \xrightarrow[n\rightarrow +\infty]{a.s.}0.$$
\end{enumerate}
\item When $0<p<1$ :
\begin{enumerate}
\item If $\theta> 1$ :
$$ \frac{\ln(\tau_n)}{n}\xrightarrow[n\rightarrow +\infty]{a.s.} \alpha\ln(\theta).$$
\item  If $\theta=1$ : for any $\lambda>0$ such that $c^{\alpha\lambda}  \EE\left[|\ell_1|^{\alpha\lambda}\right] p<1$ :
$$\frac{\tau_n}{n^{1/\lambda}} \xrightarrow[n\rightarrow +\infty]{a.s.}0.$$
\end{enumerate}
\end{enumerate}
\end{thm}

Before proving Theorem \ref{theo:2}, we state a short lemma regarding the negative moments of $\xi_1$.
\begin{lem}\label{lem:NegativeMoment}
\begin{enumerate}
\item[]
\item If the underlying L\'evy process $L$ has negative jumps, then for $\lambda\geq0$ :
$$\EE[\xi_1^{-\lambda}]<+\infty  \quad \Longleftrightarrow \quad  \lambda \in [0,1).$$
\item If the underlying L\'evy process $L$ has no negative jumps, then for any $\lambda\geq0$, we have $\EE[\xi_1^{-\lambda}]<+\infty$.
\end{enumerate}
\end{lem}
\begin{proof}
We denote by $(A_t=\int_0^{t}L_u\, du, \,t\geq0)$ the (free) integrated stable L\'evy process, and by $\PP_{(x,y)}$ the law of $(A,L)$ when started from $(x,y)$. Applying the Markov property, we have for $\mu>0$ :
$$\int_0^{+\infty} e^{-\mu t}\PP_{(0,1)}\left(A_t\leq 0\right) dt = \EE\left[e^{-\mu \xi_1} \int_0^{+\infty} e^{-\mu t} \PP_{(0, \ell_1)}\left(A_t\leq 0  \right) dt  \right].$$
Recall then that $A_t \;\mathop{=}\limits^{(d)}\; \frac{t^{1+1/\alpha}}{(1+\alpha)^{1/\alpha}} L_1$ under $\PP_{(0,0)}$. Therefore, since $\ell_1<0$ a.s., we deduce that
$$1\geq \PP_{(0, \ell_1)}\left(A_t\leq 0  \right) =\PP\left(  L_1 \leq   -(1+\alpha)^{1/\alpha} t^{-1/\alpha}\ell_1  \right)  \geq 1-\rho,$$ hence
$$ \mu \int_0^{+\infty} e^{-\mu t}\PP_{(0,1)}\left(A_t\leq 0\right) dt \leq \EE\left[e^{-\mu \xi_1}\right] \leq  \frac{\mu}{1-\rho} \int_0^{+\infty} e^{-\mu t}\PP_{(0,1)}\left(A_t\leq 0\right) dt.
$$
Integrating against $\mu^{\lambda-1}$ on $(0,+\infty)$ with $\lambda>0$, we obtain :
$$\lambda \int_0^{+\infty}  t^{-\lambda-1}\PP_{(0,1)}\left(A_t\leq 0\right) dt \leq \EE\left[\xi_1^{-\lambda}\right]  \leq   \frac{\lambda}{1-\rho}  \int_0^{+\infty}  t^{-\lambda-1}\PP_{(0,1)}\left(A_t\leq 0\right) dt. $$
The result now follows from the asymptotics of the stable laws, i.e. when $L$ admits negative jumps (see Bertoin \cite[Prop.4, p221]{Ber}
$$\PP(L_1\leq - t^{-1/\alpha}) \equi_{t\rightarrow 0}   \kappa\, t$$
for some $\kappa>0$. When $L$ has no negative jumps, this asymptotics is known to be exponential.\\
\end{proof}

\begin{proof} We now come back to the proof of Theorem \ref{theo:2}.
\begin{itemize}
\item[1.(a)]
The proof of Point 1.(a) is a direct adaptation of Profeta-Simon \cite[Theorem A]{ProSim-16}, using the decomposition, for $p\geq2$:
$$\tau_{p}-\tau_{p-1} \law L_{\tau_1} \times \tau_1\times  \left(\prod_{i=1}^{p-1} c|\ell_i|  \right)^{\alpha}.$$
Note that due to the reflection, there is no need to use the dual process $-L$ here.
\item[1.(b)] Let $\varepsilon>0$. Using the Markov inequality :
\begin{align*}
\PP\left(\frac{\ln(\tau_n)}{n^{1/\lambda}}\geq \varepsilon\right)&\leq e^{-\varepsilon  n^{1/\lambda - 1/2}} \EE\left[\tau_n^{1/\sqrt{n}}\right]\\
& \leq e^{-\varepsilon  n^{1/\lambda - 1/2}}\left(\EE\left[
 \left(\xi_1U_0^{\alpha}\right)^{1/\sqrt{n}}\right]+   \sum_{k=2}^{n}  \EE\left[(\xi_k U_0^\alpha )^{1/\sqrt{n}} \right] \EE\left[  (c_{\text{crit}} |\ell_1|)^{\alpha/\sqrt{n}}\right]^{k-1}\right)\\
 &\leq  e^{-\varepsilon  n^{1/\lambda - 1/2}}\EE\left[
 \left(\xi_1U_0^{\alpha}\right)^{1/\sqrt{n}}\right] \left(1+   (n-2)  \EE\left[  (c_{\text{crit}} |\ell_1|)^{\alpha/\sqrt{n}}\right]^n \right)
\end{align*}
since, from Jensen's inequality,
$$\EE\left[(c_{\text{crit}} |\ell_1|)^{\alpha/\sqrt{n}}\right]\geq \exp\left(\frac{\alpha}{\sqrt{n}} \EE[ \ln(|\ell_1|)+\ln(c_{\text{crit}})]  \right) =1.$$
Next, using the explicit Mellin transform (\ref{eq:Utauy0}) and Taylor expansions, we may compute the limit :
\begin{align*}
\EE\left[  (c_{\text{crit}} |\ell_1|)^{\alpha/\sqrt{n}}\right]^n  =   e^{-\sqrt{n}\alpha\pi\cot\left(\pi\gamma\right)}
\left(\frac{\sin (\pi  \gamma (1+ \frac{\alpha}{\sqrt{n}}))}{\sin\left(\pi(1-\gamma) (1+ \frac{\alpha}{\sqrt{n}}) \right)}\right)^n \xrightarrow[n\rightarrow +\infty]{} \exp\left(\frac{\pi^2\alpha^2}{2} (1-2\gamma)(1+\cot^2(\pi\gamma))\right).
\end{align*}
The a.s. convergence then follows from the usual application of the Borel-Cantelli's lemma.

\item[2.(a)]\text{and 3.(a)}\; Both cases may be dealt with in the same way, by taking $p\in [0,1[$. We start with
the Markov's inequality :
\begin{align*}
\PP\left(\left|\frac{\ln(\tau_n)}{n}- \alpha\ln(\theta)\right|>\varepsilon\right) &\leq e^{-\lambda\varepsilon n} \left(\EE\left[ \tau_n^\lambda \theta^{ -\alpha \lambda n}\right]
+\EE\left[ \tau_n^{-\lambda} \theta^{\alpha\lambda n } \right]\right)
\end{align*}
Using Lemma \ref{lem:sumgn} with $\lambda$ small enough, the first term may be bounded by
$$
\EE\left[ \tau_n^\lambda \theta^{ -\alpha \lambda n }\right]\leq A_\lambda+B_\lambda \theta^{ -\alpha \lambda n } \leq A_\lambda+B_\lambda <+\infty
$$
since $\theta>1$. Similarly, since $\beta_n$ only takes the values 0 or 1, the second term may be bounded by :
\begin{align*}
\EE\left[ \tau_n^{-\lambda} \theta^{ \alpha \lambda n }\right]&\leq\theta^{ \alpha \lambda n } \EE\left[\left( \xi_{n+1}  (1-\beta_n)\theta^{\alpha n}M_n^\alpha + \xi_{n+1} \beta_n   \theta^{\alpha g_n}  M_{g_n}^\alpha \prod_{i=g_{n}+1}^{n} c^\alpha |\ell_i|^\alpha  \right)^{-\lambda} \right]\\
&\leq\theta^{ \alpha \lambda n }\left(
(1-p) \theta^{-n\alpha\lambda}\EE\left[\xi_1^{-\lambda} M_1^{-\alpha\lambda}   \right]  + p\EE[\xi_1^{-\lambda}]\EE\left[\left(\theta^{g_{n}}  M_{g_{n}}   \prod_{i=g_{n}+1}^{n} c |\ell_i|\right)^{-\alpha\lambda}\right] \right).
\end{align*}
Observe next that, decomposing the expectation with respect to the law of $g_n$ (see (\ref{gn})), we obtain
\begin{multline*}
\theta^{ \alpha \lambda n }\EE\left[\left(\theta^{g_{n}}  M_{g_{n}}   \prod_{i=g_{n}+1}^{n} c |\ell_i|\right)^{-\alpha\lambda}\right] \\=\theta^{ \alpha \lambda n } \EE[U_0^{-\alpha\lambda}] \EE[|c\ell_1|^{-\alpha\lambda}]^{n}p^{n-1} +   (1-p)\EE[M_1^{-\alpha \lambda}] \sum_{k=1}^{n-1} \theta^{\alpha \lambda k}  \EE[|c\ell_1|^{-\alpha\lambda}]^{k} p^{k-1}.
\end{multline*}
This term may be bounded by a constant independent of $n$ as soon as $\lambda$ is small enough so that
$$ \theta^{\alpha \lambda}\EE[|c\ell_1|^{-\alpha\lambda}]p<1.$$
The result then follows again from the Borel-Cantelli's lemma.

\item[2.(b)]and 3.(b) are consequences of the following result by Petrov \cite{Pet}, which we adapt here to our set-up. Assume that $(X_k, k\geq1)$ are positive r.v.'s such that $\EE[X_k^\nu]<\infty$ for some positive $\nu\leq 1$ and all $k\geq1$. If $A_n = \sum_{k=1}^n \EE[X_k^\nu] \xrightarrow[n\rightarrow+\infty]{} +\infty$, then for any $0<\lambda<\nu$, $\sum_{k=1}^n X_k = o( A_n^{1/\lambda})$ a.s. \\
\noindent
We therefore apply the aforementioned result with $X_k = \tau_k-\tau_{k-1}$. When $p=0$ and $\theta=1$, we choose  $\displaystyle \nu<\frac{1-\rho}{1+\alpha\rho}$. This yields :
$$\sum_{k=2}^n \EE[(\tau_k-\tau_{k-1})^\nu] = \sum_{k=1}^n \EE[(\xi_kM_{k-1}^\alpha)^\nu]\equi_{n\rightarrow +\infty}  \EE[(\xi_1M_1^\alpha)^\nu]\, n.$$
When $p\in (0,1)$ and $\theta=1$, we choose $\nu>0$ such that  $c^{\alpha\nu}  \EE\left[|\ell_1|^{\alpha\nu}\right] p<1$. This yields :
$$
\sum_{k=2}^n \EE[(\tau_k-\tau_{k-1})^\nu] =(1-p) \EE[(\xi_1M_1^\alpha)^\nu]\, (n-2) + p \EE\left[\xi_1^\nu\right] \sum_{k=2}^n \EE\left[ \left(  M_{g_{k-1}}   \prod_{i=g_{k-1}+1}^{k-1} c |\ell_i|\right)^{\alpha\nu} \right].\\
$$
This last term may be computed by letting  $\theta\rightarrow 1$ in Lemma \ref{lem:sumgn}. We obtain :
\begin{multline*}
\sum_{k=1}^{n-1} \EE\left[ \left(  M_{g_{k}}   \prod_{i=g_{k}+1}^{k} c |\ell_i|\right)^{\alpha\nu} \right]
= \frac{1}{p}\EE[U_0^{\alpha\nu}] \left(c^{\alpha\nu}\EE\left[ |\ell_1|^{\alpha\nu}\right] p\right)^{n-1}
\\+ \sum_{i=1}^{n-2} \frac{1}{p} \left(c^{\alpha\nu }\EE\left[ |\ell_1|^{\alpha\nu}\right] p \right)^i \left(
\EE[U_0^{\alpha\nu}]  + (1-p)  \EE[M_1^{\alpha\nu}] (n-1-i)\right).
\end{multline*}
Letting $n\rightarrow +\infty$, we obtain the asymptotics :
$$  \sum_{k=1}^{n-1} \EE\left[ \left(  M_{g_{k}}   \prod_{i=g_{k}+1}^{k} c |\ell_i|\right)^{\alpha\nu} \right] \equi_{n\rightarrow +\infty}
n \times (1-p)  \EE[M_1^{\alpha\nu}]  \sum_{i=1}^{+\infty} \frac{1}{p} \left(c^{\alpha\nu }\EE\left[ |\ell_1|^{\alpha\nu}\right] p \right)^i$$
and Points 2.(b) and 3.(b) thus follow directly from Petrov's result.
\end{itemize}

\end{proof}

\section{Langevin processes conditioned of not hitting $(0,0)$}\label{sec:3}

We shall construct in this section the law of an integrated $\alpha$-stable L\'evy process conditioned to stay positive, thus extending some earlier results by Groeneboom, Jongbloed and Wellner \cite{GJW} on integrated Brownian motion. Note that a direct construction seems difficult as we do not have the exact asymptotic of $\PP_{(x,y)}(\tau_1>t)$ but only lower and upper bounds, see \cite[Theorem A]{ProSim-15}. We assume in this section that $p=1$, and $c<c_{\text{crit}}$ so that $\tau_\infty<+\infty$ a.s. We now denote by $\PP^{(c)}$ the law of the solution of (\ref{eq:ConfinedModel}), i.e. of  the integrated $\alpha$-stable Langevin process reflected on a partially elastic boundary,  and, to simplify $\PP=\PP^{(0)}$. The general idea of this section is to first condition the process $(X,U)$ under $\PP^{(c)}$ to not hit the boundary $(0,0)$, which is done using a renewal result, and then to let $c\rightarrow 0$. As a consequence of this construction, we shall finally obtain the exact asymptotics of $\PP_{(x,y)}(\tau_1>t)$.
\\

\noindent
We start by observing that the law of  $U_{\tau_1^-}$ is the same under $\PP^{(c)}$ for any $c\geq0$. Its Mellin's transform is given as follows,  see (\cite[Formula 2.1]{ProSim-16}). Let $q$ be the stable density whose Fourier transform is given by :
$$\int_\RR e^{i\lambda z} q(z) dz =\EE\left[\exp\left(  i\lambda \int_0^1 L_sds\right)\right] = \exp\left( -\;\frac{1}{\alpha +1}\,(i \lambda)^\alpha e^{-i\pi\alpha\rho\, {\rm sgn}(\lambda)}\right).$$
For $s\in (0,1)$, we consider the function
$$\Psi_s(x,u) =  \frac{1}{\Gamma(1-s)}\iint_0^{\infty} \! \lambda^{ -s} q\left(-\left(1 +x\lambda^{1+\alpha} +ut\lambda \right)t^{-1-\frac{1}{\alpha}} \right) d\lambda \, t^{-1-1/\alpha} dt.$$
This function admits an analytic continuation for $s\in[0, \frac{1}{1-\gamma}]$ which we denote $\widehat{\Psi}_s(x,u)$. Then, for $\nu\in\left(0,  \frac{\gamma}{1-\gamma}\right)$, the Mellin's transform of $U_{\tau_1^-}$  admits the expression :
\begin{equation}\label{eq:Utau}
\EE_{(x,u)}\left[|U_{\tau_1^-}|^{\nu}\right] =  \frac{\pi (1+\alpha)^{\frac{\alpha-\nu}{1+\alpha}}  }{\Gamma^2\left(\frac{1+\nu}{1+\alpha}\right)\sin\left(\pi (1+\nu) (1-\gamma)\right)} \widehat{\Psi}_{\nu+1}( x,u).
\end{equation}

\begin{rem}
In the following, to avoid complicated notations in conditional expectations, we shall systematically remove the superscript $^{(c)}$ when taking the expectation of $\mathcal{F}_{\tau_1}$-measurable random variables. Therefore, in the following proof, the notations $\EE^{(c)}$ and $\PP^{(c)}$ will always apply to the variables $\tau_{\infty}$, $\tau_n$ and $U_{\tau_n^-}$ with $n\geq 2$, while the notations $\EE$ and $\PP$ will refer only to  $\tau_1$ and $U_{\tau_1^-}$.
\end{rem}

 \subsection{The case $0<c<c_{\text{crit}}$}

\noindent
We follow Jacob \cite[Section 3]{Jacob-13}. Let $\eta(c)>0$ be the unique solution of the equation
\begin{equation}\label{eq:kc}
c^{\alpha \eta(c)}\EE_{(0,1)}\left[|U_{\tau_1^-}|^{\alpha \eta(c)}\right]=1\qquad \Longleftrightarrow \qquad c^{\alpha \eta(c)}\frac{\sin\left(\pi \gamma(\alpha \eta(c)+1)\right)  }{\sin\left(\pi(1-\gamma)(\alpha \eta(c)+1)\right) }=1.
\end{equation}
Note that $\eta(c)$ is well-defined since $f(x) = \EE_{(0,1)}\left[|cU_{\tau_1^-}|^{\alpha x}\right]$ is a convex function, whose derivative at $x=0$ is negative, and such that $ \lim\limits_{x\rightarrow \frac{1-\rho}{1+\alpha \rho}} f(x) = +\infty$. This implies in particular that  $\eta(c)$ is a decreasing function of $c$  such that
$$\displaystyle\lim_{c\rightarrow 0}\eta(c) = \frac{1-\rho}{1+\alpha \rho}=\eta \qquad \text{and}\qquad \displaystyle\lim_{c\rightarrow c_{\text{crit}}}\eta(c) =0.$$
We define the harmonic function $h^c$ for $\{x>0 \text{ and }u\in \RR\}$ or  $\{x=0 \text{ and }u>0\}$ by
\begin{equation}\label{eq:hc}
h^c(x,u) = c^{\alpha \eta(c)}\EE_{(x,u)}[|U_{\tau_1^-}|^{\alpha \eta(c)}]=\EE^{(c)}_{(x,u)}\left[ |U_{\tau_1}|^{\alpha \eta(c)} \right]. 
\end{equation}
Note that $h^c$ enjoys the following scaling property  :
\begin{equation}\label{eq:scaleh}
h^c(x,u) =  x^{\frac{\alpha \eta(c)}{1+\alpha}}h^c(1,u x^{-1/(\alpha+1)}).
\end{equation}
In particular, for $x> 0$ and $u> 0$, we have
$$h^c(0,u) =u^{\alpha \eta(c)} \qquad \text{ and }\qquad h^c(x,0) = x^{\frac{\alpha \eta(c)}{1+\alpha}} h^c(1,0).$$

\begin{prop}
For $0<c<c_{\text{crit}}$, there exists a probability $\PP_{(x,u)}^{(c)\uparrow}$ on $(\Omega, \Ff_\infty)$ such that
$$\forall \Lambda_s\in \Ff_s,\qquad \lim_{t\rightarrow+\infty}\PP_{(x,u)}^{(c)}(\Lambda_s| \tau_\infty>t) = \PP_{(x,u)}^{(c)\uparrow}(\Lambda_s).$$
$\PP_{(x,u)}^{(c)\uparrow}$ may be described by an $h$-transform with respect to $\PP_{(x,u)}^{(c)}$ as follows :
$$\forall \Lambda_s\in \Ff_s, \qquad\PP_{(x,u)}^{(c)\uparrow}(\Lambda_s) = \frac{1}{h^c(x,u)} \EE^{(c)}_{(x,u)}\left[\ind_{\Lambda_s} h^c(X_s, U_s) \ind_{\{s<\tau_\infty\}}\right]. $$
\end{prop}

\begin{proof}
Using the renewal Equation (\ref{eq:10bis}) and the scaling property, we deduce from Goldie \cite[Theorem 4.1]{Gol}, that there exists $\kappa>0$ such that:
$$\PP^{(c)}_{(0,u)}(\tau_\infty>t) \equi_{t\rightarrow +\infty}  \kappa \frac{u^{\alpha \eta(c)}}{t^{\eta(c)}}$$
where $\eta(c)$ is the solution of Equation (\ref{eq:kc}).
Applying the Markov property at the time $\tau_1$, we then obtain that
$$\PP^{(c)}_{(x,u)}(\tau_\infty>t) = \EE_{(x,u)}\left[\PP^{(c)}_{(0, c|U_{\tau_1^-}|)}(\tau_\infty >t-\tau_1)\right].
 $$
 To apply the dominated convergence theorem, let us fix some deterministic $A>0$ such that, for any $t\geq A$, we have $t^{\eta(c)}\PP_{(0, 1)}^{(c)}(\tau_\infty >t)\leq 2\kappa$. Then, by scaling and since $0<\eta(c)<1$ :
 \begin{align*}
 &t^{\eta(c)}\PP_{(0, c|U_{\tau_1^-}|)}^{(c)}(\tau_\infty >t-\tau_1) \\
 \leq&\; (t-\tau_1)^{\eta(c)}\PP_{(0, 1)}^{(c)}\left(\tau_\infty >\frac{t-\tau_1}{c^\alpha|U_{\tau_1^-}|^\alpha}\right) + \tau_1^{\eta(c)}\\
 \leq&\;
A^{\eta(c)}c^{\alpha \eta(c)} |U_{\tau_1^-}|^{\alpha \eta(c)}  \ind_{\left\{\frac{t-\tau_1}{c^\alpha|U_{\tau_1^-}|^\alpha}\leq A\right\}} + 2\kappa c^{\alpha \eta(c)} |U_{\tau_1^-}|^{\alpha \eta(c)}  \ind_{\left\{\frac{t-\tau_1}{c^\alpha|U_{\tau_1^-}|^\alpha}\geq A\right\}} + \tau_1^{\eta(c)}\\
\leq &\; (A^{\eta(c)}+2\kappa)c^{\alpha \eta(c)} |U_{\tau_1^-}|^{\alpha \eta(c)}  + \tau_1^{\eta(c)}
 \end{align*}
 which is integrable since $\eta(c)<\eta$. The dominated convergence theorem then yields :
 $$\PP^{(c)}_{(x,u)}(\tau_\infty>t) \equi_{t\rightarrow +\infty}\kappa  \frac{h^c(x,u)}{t^{\eta(c)}}.$$
Next, applying the Markov property at time $s$, we deduce that
$$\frac{\PP^{(c)}_{(x,u)}(\tau_\infty>t  | \mathcal{F}_s)}{\PP^{(c)}_{(x,u)}(\tau_\infty>t  )} \xrightarrow[t\rightarrow +\infty]{} \frac{h^c(X_s,U_s)}{h(x,u)}\ind_{\{\tau_\infty>s\}}$$
and the result (i.e. the $L^1$ convergence) will follow from Scheff\'e's lemma, once we have proven that
$$\EE^{(c)}_{(x,u)}\left[h^c(X_s,U_s)\ind_{\{\tau_\infty>s\}}\right]=h^c(x,u).$$
Observe that by definition of $\eta(c)$:
\begin{align*}
 \EE^{(c)}_{(x,u)}[|U_{\tau_n^-}|^{\alpha \eta(c)}] &=  \EE^{(c)}_{(x,u)}\left[  \EE_{(0, c|U_{\tau_{n-1}^-}|)} \left[|U_{\tau_1^-}|^{\alpha \eta(c)}\right]\right] \\
 &= \EE^{(c)}_{(x,u)}\left[   (c|U_{\tau_{n-1}^-}|)^{\alpha \eta(c)}  \EE_{(0,1)} \left[|U_{\tau_1^-}|^{\alpha \eta(c)}\right]\right] = \EE^{(c)}_{(x,u)}[|U_{\tau_{n-1}^-}|^{\alpha \eta(c)}].
\end{align*}
By iteration, we deduce that
\begin{align}
\notag h^c(x,u) &=  c^{\alpha \eta(c)}\EE^{(c)}_{(x,u)}[|U_{\tau_n^-}|^{\alpha \eta(c)}]\\
\label{eq:hctau}&=c^{\alpha \eta(c)} \EE^{(c)}_{(x,u)}\left[|U_{\tau_n^-}|^{\alpha \eta(c)} \ind_{\{\tau_n\leq s\}} \right] +  \EE^{(c)}_{(x,u)}\left[h^c(X_s, U_s)\ind_{\{\tau_n> s\}} \right].
\end{align}
It remains to prove that the first term converges towards 0 as $n \rightarrow +\infty$.

By the Markov property :
\begin{multline*}
I_n:=\EE^{(c)}_{(x,u)}\left[|U_{\tau_n^-}|^{\alpha \eta(c)} \ind_{\{\tau_n\leq s\}} \right]  \\
=\underbrace{ \EE^{(c)}_{(x,u)}\left[|U_{\tau_n^-}|^{\alpha \eta(c)} \ind_{\{\tau_n\leq s\}}\ind_{\{|U_{\tau_n^-}|\leq 1\}} \right]}_{J_n}
+\underbrace{ \EE^{(c)}_{(x,u)}\left[   \EE_{(0, c|U_{\tau_{n-1}^-}|)}\left[|U_{\tau_1^-}|^{\alpha \eta(c)} \ind_{\{\tau_1+ \tau_{n-1}\leq s\}} \ind_{\{|U_{\tau_1^-}|\geq 1\}}   \right]\right]}_{K_n}.
\end{multline*}
Observe first that by dominated convergence, $J_n\xrightarrow[n\rightarrow +\infty]{}0$. Next, by scaling, the second term $K_n$ may be written
$$K_n=\EE^{(c)}_{(x,u)}\left[ (c|U_{\tau_{n-1}^-}|)^{\alpha \eta(c)}  \EE_{(0,1)}\left[|U_{\tau_1^-}|^{\alpha \eta(c)} \ind_{\{\tau_1c^\alpha|U_{\tau_{n-1}^-}|^\alpha+ \tau_{n-1}\leq s\}} \ind_{\{ c |U_{\tau_{n-1}^-}| |U_{\tau_1^-}|\geq 1\}}   \right]\right].
$$
Applying the inequality  $\ind_{\{a+b\leq s\}} \leq \ind_{\{a\leq s\}}\ind_{\{b\leq s\}}$ which is valid for positive $a$ and $b$, we obtain
$$K_n\leq \EE^{(c)}_{(x,u)}\left[ |U_{\tau_{n-1}^-}|^{\alpha \eta(c)}\ind_{\{ \tau_{n-1}\leq s\}}   \EE_{(0,1)}\left[(c|U_{\tau_1^-}|)^{\alpha \eta(c)} \ind_{\{\tau_1c^\alpha|U_{\tau_{n-1}^-}|^\alpha\leq s\}}\ind_{\{ c |U_{\tau_{n-1}^-}| |U_{\tau_1^-}|\geq 1\}}   \right]\right]. $$
Then, since $$\ind_{\{\tau_1c^\alpha|U_{\tau_{n-1}^-}|^\alpha\leq s\}}\ind_{\{ c |U_{\tau_{n-1}^-}| |U_{\tau_1^-}|\geq 1\}}\; \leq\; \ind_{\{\tau_1/|U_{\tau_{1}^-}|^\alpha\leq s\}} $$
we deduce that
$$K_n\leq \; \EE^{(c)}_{(x,u)}\left[ |U_{\tau_{n-1}^-}|^{\alpha \eta(c)}\ind_{\{ \tau_{n-1}\leq s\}} \right]   \EE_{(0,1)}\left[(c|U_{\tau_1^-}|)^{\alpha \eta(c)} \ind_{\{\tau_1/|U_{\tau_{1}^-}|^\alpha\leq s\}} \right] = I_{n-1} \times r
$$
where we have set $r =  \EE_{(0,1)}\left[(c|U_{\tau_1^-}|)^{\alpha \eta(c)} \ind_{\{\tau_1/|U_{\tau_{1}^-}|^\alpha\leq s\}} \right] \in (0,1)$. By iteration, we obtain for $n\geq 2$,
$$I_n \leq \sum_{k=0}^{n-2}  J_{n-k}\, r^k  + I_1 r^{n-1}$$
and the result follows by letting $n\rightarrow +\infty$ and using dominated convergence.

\end{proof}

\subsection{The case $c=0$}

\noindent
We are now interested in letting $c\rightarrow 0$, in order to obtain the law of a (free) stable Langevin process conditioned on not hitting $0$. In this case, notice from (\ref{eq:kc}) and the limit of $\eta(c)$ that :
\begin{align*}
\lim_{c\rightarrow 0} h^c(x,u)= \lim_{\eta(c) \rightarrow \eta}    \frac{\sin\left(\pi (1+\alpha \eta(c))(1-\gamma)\right)}{\sin (\pi  \gamma(1+\alpha \eta(c)))} \EE_{(x,u)}[|U_{\tau_1^-}|^{\alpha \eta(c)}].
\end{align*}
Passing to the limit in the expression (\ref{eq:Utau}) of $h^c$, we deduce that $h^0$ admits the representation :
\begin{align*}\label{eq:h0}
h^0(x,u) =  \frac{\pi (1+\alpha)^{\frac{\alpha(1-\eta)}{1+\alpha}}  }{\Gamma^2\left(\frac{1+\alpha \eta}{1+\alpha}\right)\sin (\pi  \gamma(1+\alpha \eta))}\widehat{\Psi}_{\alpha \eta+1}( x,u).
\end{align*}
Note also, that writing
$$\EE_{(x,u)}[|U_{\tau_1^-}|^{\alpha \eta(c)}] = h^c(x,u)  \frac{\sin (\pi  \gamma(1+\alpha \eta(c)))} {\sin\left(\pi (1+\alpha \eta(c)) (1-\gamma)\right)}$$
$h^0$ may  be obtained by the converse mapping theorem for Mellin's transform (see for instance \cite{FGD}) :
$$\PP_{(x,u)}(|U_{\tau_1^-}|>z) \equi_{z\rightarrow +\infty}  \frac{\sin\left(\pi \gamma(1+\alpha \eta)\right)  }{\pi \gamma}\frac{h^0(x,u)}{z^{\alpha \eta}}.$$
We will see in Corollary \ref{cor:persistence} that unlike $h^c$, the function $h^0$ is increasing in both variables, and that for $x>0$,  $\lim\limits_{u\rightarrow -\infty} h^0(x,u)=0$. The insight for this is as follows : for $U_{\tau_1^-}$ to take a very large negative value, the process $U$ must first make a very long positive excursion (so that $X$ take a large positive value) before dropping to the negative values. In other words, if the starting point $U_0=u$ is very negative, then $X$ will hit almost immediately the boundary, i.e. $U_{\tau_1^-}\simeq u$, and thus the probability that  $U_{\tau_1^-}\simeq u$ is smaller than $-z$ will be close to 0 as $z\rightarrow +\infty$.

\begin{cor}\label{cor:Pup0}
The law of an integrated $\alpha$-stable L\'evy process conditioned to stay positive is given by
$$\forall \Lambda_s\in \Ff_s, \qquad\PP_{(x,u)}^{\uparrow}(\Lambda_s) = \frac{1}{h^0(x,u)} \EE_{(x,u)}\left[\ind_{\Lambda_s} h^0(X_s, U_s) \ind_{\{s<\tau_1\}}\right] $$
\end{cor}
\begin{proof}
To show that this definition makes sense, we shall prove that $\PP^\uparrow$ may be obtained by a penalization procedure, i.e. that :
$$\PP_{(x,u)}\left(\Lambda_s |  |U_{\tau_1^-}|>z \right) \xrightarrow[z\rightarrow +\infty]{} \PP_{(x,u)}^\uparrow\left(\Lambda_s\right).
$$
Indeed, observe first that :
\begin{align*}
z^{\alpha \eta}\PP_{(x,u)}\left(|U_{\tau_1^-}|>z  | \mathcal{F}_s \right) &= z^{\alpha \eta}\PP_{(x,u)}\left(\{|U_{\tau_1^-}|>z\} \cap  \{\tau_1\leq s\}  | \mathcal{F}_s  \right)  +   z^{\alpha \eta}\PP_{(x,u)}\left(\{|U_{\tau_1^-}|>z\} \cap  \{\tau_1> s\}  | \mathcal{F}_s  \right)\\
&= z^{\alpha \eta}\ind_{\{\{|U_{\tau_1^-}|>z\} \cap  \{\tau_1\leq s\}  \}} + z^{\alpha \eta}\ind_{\{\tau_1 > s  \}}\PP_{(X_s, U_s)} \left( |U_{\tau_1^-}|>z \right)\\
&\xrightarrow[z\rightarrow +\infty]{}  \ind_{\{\tau_1 > s  \}}h^0(X_s, U_s)  \frac{\sin\left(\pi \gamma(1+\alpha \eta)\right)  }{\pi \gamma}
\end{align*}
hence, as before, the $L^1$-convergence will follow from Scheff\'e's lemma once we have proven that
$$ \EE_{(x,u)}\left[h^0(X_s, U_s) \ind_{\{s<\tau_1\}}\right] = h^0(x, u). $$
Going back to Formula (\ref{eq:hctau}) with $n=1$, we obtain that :
\begin{equation}
h^c(x,u)
\label{eq:hc0}=c^{\alpha \eta(c)} \EE_{(x,u)}\left[|U_{\tau_1}|^{\alpha \eta(c)} \ind_{\{\tau_1\leq s\}} \right] +  \EE_{(x,u)}\left[h^c(X_s, U_s)\ind_{\{\tau_1> s\}} \right].
\end{equation}
Since $\eta \alpha<\alpha$, the first term is easily bounded by :
\begin{align*}
c^{\alpha \eta(c)} \EE_{(x,u)}\left[|U_{\tau_1}|^{\alpha \eta(c)} \ind_{\{\tau_1\leq s\}} \right]
&\leq c^{\alpha \eta(c)} \EE_{(x,u)}\left[\sup_{u\leq s}|U_{u}|^{\alpha \eta(c)} \ind_{\{\tau_1\leq s\}} \right]\\
&\leq c^{\alpha \eta(c)} \EE_{(x,u)}\left[\sup_{u\leq s}|U_{u}|^{\alpha \eta}\wedge1\right]\xrightarrow[c\rightarrow 0]{}0.
\end{align*}

\noindent
Next, fix $0<\delta<c_{\text{crit}}$ and observe that we may find two constants $A_\delta$ and $B_\delta$, such that,  for any $c\in [0,\delta]$, we have :
$$h^c(1,u) \leq A_\delta\, +B_\delta\, (|u|\wedge1)^{\alpha\eta} .$$
By the scaling property of $h^c$, we then obtain :
$$h^c(x,u) \leq A_\delta\,(x\wedge1)^{\frac{\alpha\eta}{1+\alpha}} + B_\delta\, (|u|\wedge1)^{\alpha\eta}. $$
The result finally follows by passing to the limit in  (\ref{eq:hc0}) and using the dominated convergence theorem, since $X_s$ and $L_s$ admit moments of order $\alpha-\varepsilon$ under $\PP$.

\noindent

\end{proof}

\begin{cor}\label{cor:persistence}
There exists a constant $\kappa>0$ such that :
$$\PP_{(x,u)}(\tau_1>t)\equi_{t\rightarrow +\infty} \kappa\frac{h^0(x,u)}{t^{\eta}}.$$
\end{cor}

\begin{proof}

Using Corollary \ref{cor:Pup0}, we first have :
$$h^0(x,u) \EE^{\uparrow}_{(x,u)}\left[ \frac{1}{h^0(X_t, U_t)}\right] = \PP_{(x,u)}(\tau_1>t).$$
By the scaling property of $h^0$ and $(X_t, U_t)$, we deduce that :
$$\PP_{(x,u)}(\tau_1>t) =  \frac{h^0(x,u)}{t^{\eta}} \EE^{\uparrow}_{\left(\frac{x}{t^{1+1/\alpha}},\frac{u}{t^{1/\alpha}}\right)}\left[ \frac{1}{h^0(X_1, U_1)}\right].$$
Letting $t\rightarrow +\infty$, we finally obtain that
$$\lim_{t\rightarrow +\infty}t^{\eta} \PP_{(x,u)}(\tau_1>t) = h^0(x,u)\EE^{\uparrow}_{\left(0,0\right)}\left[ \frac{1}{h^0(X_1, U_1)}\right] $$
which is finite from  \cite[Theorem A]{ProSim-15}.
 
\end{proof}

\begin{rem}
It is clear from Corollary \ref{cor:persistence} that $h$ is increasing in both variables. Furthermore, using exponents for the starting points, we have for $x>0$  and $u\leq 0$:
$$h(x,u) = \EE\left[ h^0\left( X_1^{(x,u)} , L_1^{(x,u)} \right) 1_{\{1<\tau_1^{(x,u)}\}}  \right] \xrightarrow[u\rightarrow -\infty]{}0$$
by dominated convergence since
$$h^0\left( X_1^{(x,u)} , L_1^{(x,u)} \right) 1_{\{1<\tau_1^{(x,u)}\}} \leq h^0\left( X_1^{(x,0)} , L_1^{(x,0)} \right) 1_{\{1<\tau_1^{(x,0)}\}}.$$
Note also that, had we known this asymptotics before, the conditioning of $X$ not to hit the boundary 0 would have been more direct.
\end{rem}

\begin{rem}
We briefly check that in the Brownian set-up, our approach agrees with the existing formulae in the literature. In \cite{GJW}, working directly with the explicit density of $\PP_{(x,u)}(\tau_1\in dt)$, the authors find the harmonic function :
$$\widehat{h}(x,u) = \iint_0^{+\infty}  w^{3/2}\left(q_t(x,u; 0,-w) - q_t(x,u; 0,w)\right)\, dt\, dw$$
where $q_t$ denotes the density of the Brownian motion $B_t$ and its integral :
$$ q_t(x,u; y,v)\, dy\, dv=\PP\left(x+ut +\int_0^t B_s ds \in dy,\; u+B_t \in dv \right).$$
Looking now at Lachal \cite[Formule (17)]{Lachal-97}, and replacing the + by a - in the definition of $\pi_t$, we observe that for $|s|<\frac{3}{2}$ :
$$\EE_{(x,u)}\left[|B_{\tau_1^-}|^{s-1}\right] = \left(\frac{1}{2\cos\left(\frac{\pi s}{3}\right)}-1\right)\iint_0^{+\infty}  w^{s}\left(q_t(x,u; 0,-w) - q_t(x,u; 0,w)\right)\, dt\, dw.$$
The inverse mapping for the Mellin's transform yields thus :
$$\PP_{(x,u)}\left(|B_{\tau_1^-}|>z\right) \equi_{z\rightarrow +\infty}  \frac{3}{\pi}\frac{\widehat{h}(x,u)}{z^{1/2}}$$
which agrees with the approach we use in Corollary \ref{cor:Pup0}.
Note that due to our normalization  $U=\sqrt{2}B$, we have in fact :
$$\widehat{h}(x,u) = \frac{1}{2^{1/4}} h^0(\sqrt{2} x,\sqrt{2}u).$$

\end{rem}

\section{Link with kinetic equations and a probabilistic approach for related trace problems}\label{sec:4}
 In this section, we apply the results of Section \ref{sec:2} in order to exhibit the link between \eqref{eq:ConfinedModel} and the trace problems related to kinetic equations endowing Maxwellian boundary conditions (see e.g. \cite{CerIllPul-94}, Mischler \cite{Mischler-10}).

 \subsection{The Brownian case}
 The link between the sequence of zero times of the integrated Brownian motion, the modeling of boundary conditions for Langevin dynamics and trace problems for kinetic equations was previously exploited in %Bossy and Jabir
 \cite{BoJa-11} (see also \cite{BoJa-15} for the multi-dimensional case) in order to show the well-posedness of some Lagrangian Stochastic model related to wall-bounded fluid flows.
The trace problem related to a simple Langevin model driven by a one-dimensional Brownian diffusion ($L=\sqrt{2}B$) and endowing purely reflective boundary conditions ($p=c=1$) concerns the existence, in an appropriate sense, of a solution to the boundary value problem:

\begin{subequations}
\begin{equation}\label{eq:FokkerPlanckmd}
\partial_t\rho(t,x,u)+u\partial_x \rho(t,x,u)-\partial^2_u\rho(t,x,u)=0,\quad (t,x,u)\in (0,\infty)\times(0,\infty)\times \RR,
\end{equation}
\begin{equation}\label{eq:Maxwellmd}
\rho(t,0,u)=\rho(t,0,-u),\quad (t,u)\in(0,\infty)\times\RR,
\end{equation}
\end{subequations}
where $\rho(t)$ represents the probability density function of $(X_t,U_t)$.
In a rigorous way, the variational formulation of \eqref{eq:FokkerPlanckmd}-\eqref{eq:Maxwellmd} consists in the existence of $\rho$ and the existence of a pair of trace functions $\gamma^{+}(\rho)$ and $\gamma^{-}(\rho)$ defining  the value of $\rho(t,0,u)$ along the respective boundary sets
\[
\Sigma^+=\left\{(t,u)\in(0,\infty)\times\RR\,|\,u<0\right\}\,\,
\mbox{and}\,\,\Sigma^-=\left\{(t,u)\in(0,\infty)\times\RR\,|\,u>0\right\}
\]
and such that: for all $0\leq T<\infty$ and for all $f\in\Cc^\infty_c((0,T)\times[0,\infty)\times\RR)$,
\begin{equation}\label{eq:Variationalmd}
\begin{aligned}
&\int_0^T\iint_{(0,\infty)\times\RR}\left(\partial_tf(t,x,u)+u \partial_xf(t,x,u)+\partial^2_u f(t,x,u)\right)\rho(t,x,u)\,dt\,dx\,du\\
&=-\iint_{\Sigma^+}u\gamma^+(\rho)(t,0,u)f(t,0,u)\ind_{\{0\leq t\leq T\}}\,dt\,du -\iint_{\Sigma^-}u\gamma^-(\rho)(t,0,u)f(t,0,u)\ind_{\{0\leq t\leq T\}}\,dt\,du\\
\end{aligned}
\end{equation}

\noindent
From a PDE point of view, the existence of trace functions can be handled in a classical sense by showing the continuity of $x\mapsto \rho(t,x,u)$ up to the axis $x=0$ or in a weak sense by showing some appropriate Sobolev estimates for $\rho$. As noticed in \cite{BoJa-11}, the trace functions $\gamma^+$ and $\gamma^-$ have also a natural probabilistic interpretation as density functions related to $\sum_{n\geq 1} \PP\circ(\tau_n,U_{\tau_n})^{-1}$ for the solution of the SDE :
\begin{equation*}
\left\{
\begin{aligned}
&X_t=X_0+\int_0^t U_s\,ds,\quad U_t=U_0+\sqrt{2}B_t-2\sum_{n\geq 1}U_{\tau_n^-}\ind_{\{\tau_n\leq t\}},\\
&\tau_n=\inf\{t\geq \tau_{n-1};\,X_t=0\},\,\tau_0=0.
\end{aligned}
\right.
\end{equation*}
\subsection{The stable Langevin case}
In the more general case of the stable Langevin model \eqref{eq:ConfinedModel} and assuming that $\theta=1$, the probabilistic interpretation of the trace functions $\gamma^\pm$ in terms of the SDE
\begin{equation}
\label{eq:SDE2}
\left\{
\begin{aligned}
&X_t=X_0+\int_0^t U_s\,ds,\\
&U_t=U_0+L_t+\sum_{n\geq 1}\left((1-\beta_n)(M_n-U_{\tau_n^-})-(1+c)\beta_n U_{\tau_n^-}\right)\ind_{\{\tau_n\leq t\}},\\
&\tau_n=\inf\{t\geq \tau_{n-1};\,X_t=0\},\,\tau_0=0,
\end{aligned}
\right.
\end{equation}
is given by
\begin{thm}\label{thm:GeneralTrace} Assume that
\begin{equation}
\PP(M_1\in du)= u m(u)\,du,\label{hyp:Maxwell}
\end{equation}
and that the following properties hold true:
\begin{align*}
(P_1)&\quad \forall\,n\geq 1,\,\PP\circ(\tau_n,U_{\tau_n^-})^{-1}\,\mbox{is absolutely continuous w.r.t. the measure }\,\left(\ind_{\{0\leq t\leq T\}}dt\right)\otimes\left( u\ind_{\{u\leq 0\}} du\right),\\
(P_2)&\quad \text{For all }0<T<\infty,\quad\sum_{n\geq 1}\,\PP(\tau_n\leq T)<\infty.
\end{align*}
Then there exists a non-negative integrable Borel function $\gamma^+$ defined on $\Sigma^+$ such that, for $\mu_t(dx,du)=\PP(X_t\in dx, U_t\in du)$, we have: For all $\Cc^\infty_c([0,T)\times[0,\infty)\times\RR)$-scalar function $f$,
\begin{equation}\label{eq:Variational1d}
\begin{aligned}
&\int_0^T\iint_{(0,\infty)\times\RR}\left(\partial_tf(t,x,u)+u\partial_x f(t,x,u)+\partial^{\alpha}_uf(t,x,u)\right)\mu_t(dx,du)\,dt\\
&=-\iint_{(0,\infty)\times\RR}f(0,x,u)\mu_0(dx,du)\,du-\iint_{\Sigma^+}u\gamma^+(t,0,u)f(t,0,u)\ind_{\{0\leq t \leq T\}}dtdu\\
&\quad-\iint_{\Sigma^-}u\gamma^-(t,0,u)f(t,0,u)\ind_{\{0\leq t \leq T\}} dt du
\end{aligned}
\end{equation}
where $\partial^{\alpha}$ is the fractional Laplace operator:
\[
\partial^{\alpha}_uf(u):=C(\alpha)\int_{\{y\neq 0\}}\frac{f(y+u)-f(u)-yf'(u)\ind_{\{|y|\leq 1\}}}{|y|^{\alpha+1}}\,dy,
\]
and
\begin{equation}\label{eq:tracecond}
\gamma^-(t,0,u)=\frac{p}{c^2}\gamma^+\left(t,0,\frac{-u}{c}\right)+(1-p) m(u)\left(-\int_{\{v\leq 0\}}v\gamma^+(t,0,v)\,dv\right).
\end{equation}
\end{thm}
\begin{proof} For $f\in\Cc^\infty_c([0,T)\times[0,\infty)\times\RR)$, It\^o's formula immediately yields that (see e.g. Protter [\cite{Protter-04}, Chapter $2$, Theorem $32$])
\begin{align*}
0&=\EE\left[f(0,X_0,U_0)\right]+\EE\left[\int_0^T \left(\partial_tf(t,X_t,U_t)+U_{t^-}\partial_xf(t,X_t,U_{t^-})\right)\,dt+\partial_u f(t,X_t,U_t)\,dL_t+ \partial^2_u f(t,X_t,U_{t^-})\,d\langle L\rangle^c_t\right]\\
&\quad +\EE\left[\sum_{s\leq T,\,\triangle L_s\neq 0}\left(f(s,X_s,U_s)-f(s,X_s,U_{s^-})-\partial_u f(s,X_{s},U_{s^-})\triangle L_s\right)+\sum_{t\leq T,\,X_t= 0}\left(f(t,X_t,U_t)-f(t,X_t,U_{t^-})\right)\right],
\end{align*}
separating, in the last line, the jumps related to $(L_t,\,0\leq t\leq T)$ and the jumps occurring at the zero times of $(X_t,\,0\leq t\leq T)$.

Since the infinitesimal generator of $(L_t,\,t\geq 0)$ is given by $\partial^\alpha_u$ (see e.g. [Applebaum \cite{Applebaum-04}, p. 142]), the above can be reduced to
\begin{align*}
0&=\iint f(0,x,u)\mu_0(dx,du)+\int_0^T \iint \left(\partial_tf(t,x,u)+u\partial_xf(t,x,u)+\partial^\alpha_u f(t,x,u)\right)\mu_t(dx,du)\,dt\\
&\quad +\EE\left[\sum_{n\geq1}\left(f(\tau_n,X_{\tau_n},U_{\tau_n^+})-f(\tau_n,X_{\tau_n},U_{\tau_n^-})\right)\ind_{\{\tau_n\leq T\}}\right].
\end{align*}
According to $(P_1)$ and $(P_2)$, there exists a non-negative integrable Borel function $\gamma^+$ defined on $\Sigma^+$ such that
\begin{equation}
\EE\left[\sum_{n\geq 1}f(\tau_n,X_{\tau_n},U_{\tau_n^-})\ind_{\{\tau_n\leq T\}}\right]=-\iint_{(0,T)\times\RR^-}u\gamma^+(t,0,u)f(t,0,u)\,dt\,du\label{eq:trace1}.
\end{equation}
Since 
\[
U_{\tau_n}=U_{\tau_n^-}+\triangle U_{\tau_n}=U_{\tau_n^-}+(1-\beta_n)(M_n-U_{\tau_n^-})-(1+c)\beta_nU_{\tau_n^-}=(1-\beta_n) M_n-c\beta_nU_{\tau_n^-},
\]
this implies that $\PP\circ(\tau_n,U_{\tau_n})^{-1}$ is also absolutely continuous w.r.t.
$\left(\ind_{\{0\leq t\leq T\}}dt\right)\otimes\left( u\ind_{\{u\leq 0\}} du\right)$. Denoting the related density by $\gamma^-$, we observe that,
for all $f\in\Cc_c((0,T)\times[0,\infty)\times\RR)$,
\begin{align*}
\EE\left[\sum_{n\geq 1}f(\tau_n,X_{\tau_n},U_{\tau_n})\ind_{\{\tau_n\leq T\}}\right]&=\iint_{(0,T)\times\RR^+}u\gamma^-(t,0,u)f(t,0,u)\,dt\,du\\
&=\sum_{n\geq1}\EE\left[f(\tau_n,X_{\tau_n},(1-\beta_n)M_n-\beta_ncU_{\tau_n^-})\ind_{\{\tau_n\leq T\}}\right]\\
&=(1-p)\sum_{n\geq1}\EE\left[f(\tau_n,X_{\tau_n}, M_n)\ind_{\{\tau_n\leq T\}}\right]+p\sum_{n\geq1}\EE\left[f(\tau_n,X_{\tau_n},-cU_{\tau_n^-})\ind_{\{\tau_n\leq T\}}\right].
\end{align*}
Then, since, for any $n$, the r.v.'s $\tau_n$ and $M_n$ are independent, we obtain, with $m$ the distribution of $M_1$ given by the assumption \eqref{hyp:Maxwell},
\begin{align*}
\sum_{n\geq1}\EE\left[f(\tau_n,X_{\tau_n},M_n)\ind_{\{\tau_n\leq T\}}\right]&=\sum_{n\geq1}\left(\int_0^{+\infty} u \EE\left[f(\tau_n,X_{\tau_n}, u)\ind_{\{\tau_n\leq T\}}\right]m(u)\,du\right)\\
&=\int_0^\infty  u\left(\sum_{n\geq1}\EE\left[f(\tau_n,X_{\tau_n},u)\ind_{\{\tau_n\leq T\}}\right]\right)m(u)\,du\\
&=\int_0^\infty  u\left(-\iint_{(0,T)\times\RR^-} f(t,0,u) v\gamma^+(t,0,v)\,dt\,dv\right)m(u)\,du,
\end{align*}
which implies, using (\ref{eq:trace1}),
\begin{equation}\label{eq:trace2}
\begin{aligned}
\iint_{(0,T)\times\RR^+}u\gamma^-(t,0,u)f(t,0,u)\,dt\,du&=(1-p)\left(\iint_{(0,T)\times\RR^+} u m(u)\left(-\int_{-\infty}^0 v\gamma^+(t,0,v)dv\right)f(t,0,u)\,dt\,du\right)\\
&\quad+\frac{p}{c^2}\iint_{(0,T)\times\RR^{+}} u\gamma^+\left(t,0,\frac{-u}{c}\right)f(t,0,u)\,dt\,du.
\end{aligned}
\end{equation}
Combining \eqref{eq:Variationalmd} and \eqref{eq:trace2}, we deduce that the time marginal distribution $(\mu_t,\,0\leq t\leq T)$ satisfies the variational equation \eqref{eq:Variational1d} with the boundary condition \eqref{eq:tracecond}.
\end{proof}

\begin{rem}\label{rem:trace}
Let us point out that the case $\theta\neq 1$ has been purposely left aside as this situation doesn't %refer
relate to a classical trace problem associated to Maxwell boundary conditions. %:for a kinetic equation endowing Maxwell boundary conditions.
Generically, Maxwell boundary conditions (see e.g. [Chapter $8$, \cite{CerIllPul-94}]) forms as
\begin{equation}
\label{eq:MaxwellTransition}
u\gamma^-(t,0,u)=\left(R*\gamma^+\right)(t,0,u)
,\,u>0,\,t\in [0,T],
\end{equation}
where $R=R(t,u;v)$ is a scattering kernel defining the transition law between the velocity distribution $\gamma^+$ of the outgoing particles and the velocity distribution $\gamma^-$ of the incoming particles at a time $t$.

Assuming that $\theta$ is arbitrary, and that the probabilistic representations:
\[
\iint_{(0,T)\times\RR^-}u\gamma^+(t,0,u)f(t,0,u)\,dt\,du=-\EE\left[\sum_{n\geq 1}f(\tau_n,X_{\tau_n},U_{\tau_n^-})\ind_{\{\tau_n\leq T\}}\right],
\]
\[
\iint_{(0,T)\times\RR^+}u\gamma^-(t,0,u)f(t,0,u)\,dt\,du=\EE\left[\sum_{n\geq 1}f(\tau_n,X_{\tau_n},U_{\tau_n})\ind_{\{\tau_n\leq T\}}\right],
\]
hold true, replicating the calculations for \eqref{eq:trace2} yields to
\begin{align*}
 &\iint_{(0,T)\times\RR^+}u\gamma^-(t,0,u)f(t,0,u)\ind_{\{t\leq T\}}\,dt\,du\\
&=\sum_{n\geq1}\EE\left[f(\tau_n,X_{\tau_n},(1-\beta_n)\theta^n M_n-\beta_ncU_{\tau_n^-})\ind_{\{\tau_n\leq T\}}\right]\\
&=(1-p)\sum_{n\geq1}\EE\left[f(\tau_n,X_{\tau_n},\theta^n M_n)\ind_{\{\tau_n\leq T\}}\right]+p\sum_{n\geq1}\EE\left[f(\tau_n,X_{\tau_n},-cU_{\tau_n^-})\ind_{\{\tau_n\leq T\}}\right]\\
&=(1-p)\int_0^\infty  u\left(\sum_{n\geq1}\EE\left[f(\tau_n,0,\theta^n u)\ind_{\{\tau_n\leq T\}}\right]\right)m(u)\,du+\frac{p}{c^2}\iint_{(0,T)\times\RR^{+}} u\gamma^+\left(t,0,\frac{-u}{c}\right)f(t,0,u)\ind_{\{t\leq T\}}\,dt\,du.
\end{align*}
In the case $\theta=1$,
\[
\sum_{n\geq1}\EE\left[f(\tau_n,0,\theta^n u)\ind_{\{\tau_n\leq T\}}\right]= f(t,0,u)\iint_{(0,T)\times\RR^+}v\gamma^+(t,0,v)\,dt\,dv,
\]
and we simply recover the form \eqref{eq:MaxwellTransition} which allows to formulate a usual trace problem related to kinetic systems. The case $\theta\neq 1$ doesn't provide such link as, in this case, the Maxwellian diffusive part distribution is strongly correlated to the distribution of the sequence $(\tau_n,\,n\in \mathbb{N})$, and the resulting boundary condition diverts from classical boundary value problems for kinetic equations.
\end{rem}

Owing to Theorem \ref{thm:GeneralTrace}, the trace problem related to \eqref{eq:Variational1d} is then reduced to the verification that $(P_1)$ and $(P_2)$ hold true.

\begin{thm}\label{thm:tracethm} Assume that $(L_t,\,0\leq t\leq T)$ is symmetric, that $(X_0,U_0)$ is distributed according to a probability measure $\mu_0$ defined on $(0,\infty)\times\RR$, that \eqref{hyp:Maxwell} hold true, $\theta=1$ and
\begin{description}
\item[$\bullet$] either $0\leq p <1$
\item[$\bullet$] or $p=1$ and $c>c_{\text{crit}}$ and there exists $0<\nu<1$ such that, for any $\delta \in (0,\nu)$,
\[
\iint \EE_{(x,u)}\left[|U_{\tau_1^-}|^{-\delta}\right]\mu_0(dx,du)<\infty.
\]
\end{description}
Then $(P_1)$ and $(P_2)$ hold true.
\end{thm}
As a preliminary result for the proof of Theorem \ref{thm:tracethm}, let us show that
\begin{lem}\label{lem:2}
$\PP\circ (\tau_1,U_{\tau_1^-})^{-1}$ is absolutely continuous with respect to the measure $u\ind_{\{u<0\}}du\otimes dt$.
\end{lem}
\begin{proof}[Proof of Lemma \ref{lem:2}] The idea of the proof relies on showing that  \eqref{eq:SDE2} in the case of a purely reflecting wall ($p=1$, $c=1$) admits trace functions (in a classical sense) and to deduce from \eqref{eq:trace1} that, for all $n$, $\PP\circ (\tau_n,U_{\tau_n^-})^{-1}$  is absolutely continuous with respect to $u\ind_{\{u<0\}}du\otimes dt$.\\

 \noindent
First, let us consider the distribution $\mu_t^f$ of the (free) Langevin processes
$$V_t=U_0+L_t\quad \text{ and }\quad Y_t=X_0+\int_0^t V_s\,ds.$$ For all $t>0$, $\lambda,\omega\in \RR$, we have :
\begin{align*}
\widehat{\mu}^f_t(\omega,\lambda):=\EE\left[e^{i \omega Y_t+i\lambda V_t}\right]&=\iint_{\RR\times\RR} e^{i\omega (x+ut)+i\lambda u}\EE\left[e^{i\omega \int_0^t L_s\,ds+i\lambda L_t}\right]\mu_0(dx,du)\\
&=\iint_{\RR\times\RR} e^{i\omega (x+ut)+i\lambda u}e^{-t \int_0^1\left|t\omega r+\lambda\right|^{\alpha}dr }\mu_0(dx,du)
\end{align*}
Then, the successive changes of variables $\tilde{\lambda}=\lambda/\omega$, $\tilde{\omega}^\alpha=\omega^\alpha \int_0^1|tr+\tilde{\lambda}|^\alpha \,dr$ yield
\begin{multline*}
\iint_{\RR\times\RR} |\widehat{\mu}^f_t(\omega,\lambda)| d\lambda\,d\omega\\\leq \iint_{\RR\times\RR} e^{-t\int_0^1\left|t\omega r+\lambda \right|^{\alpha}\,dr}\,d\lambda\,d\omega=\left(\int_\RR \frac{1}{\left(\int_0^1\left|rt+\tilde{\lambda} \right|^{\alpha}dr\right)^{\frac{2}{\alpha}}}\,d\tilde{\lambda}\right)\left(\int_\RR |\tilde{\omega}|e^{-t|\tilde{\omega}|^{\alpha}}\,d\tilde{\omega}\right)<\infty,
\end{multline*}
hence, for all $0<t\leq T$, the Fourier transform $\widehat{\mu}_t^f$ is integrable on $\RR\times\RR$. This implies (see e.g. Jacob and Protter \cite{JacPro-04}, Theorem 13.1) that
the distribution $\mu_{t}^f$ of $(Y_t,V_t)$ admits a bounded continuous Lebesgue density $\rho^f(t)$ on $\RR\times\RR$ given by
\[
\rho^f(t,y,v)=\frac{1}{(2\pi)^2}\iint_{\RR\times\RR} e^{-i\omega y-i\lambda v}\left(\iint_{\RR\times\RR} e^{i\omega (x+ut)+i\lambda u}e^{-t\int_0^1\left|t\omega r+\lambda\right|^{\alpha}\,dr}\mu_0(dx,du)\right)\,d\lambda\,d\omega.
\]
 Additionally, for all $k,l>1$, by applying the same change of variables as above,
\begin{align*}
 \left|\partial^{k}_y\partial^{l}_v\rho^f(t,y,v)\right|&\leq \iint_{\RR\times\RR} |\omega|^k|\lambda|^l e^{-t\int_0^1\left|t\omega r+\lambda \right|^{\alpha}\,dr}\,d\lambda\,d\omega\\
 &\leq\left(\int_\RR \frac{|\tilde{\lambda}|^l}{\left(\int_0^1|rt+\tilde{\lambda}|^{\alpha}\,dr\right)^{\frac{k+l+2}{\alpha}}}\,d\tilde{\lambda}\right)\left(\int_\RR |\tilde{\omega}|^{k+l+1}e^{-t|\tilde{\omega}|^{\alpha}}\,d\tilde{\omega}\right)<\infty,
\end{align*}
from which we deduce that, for all $t>0$, $\rho^f(t)$ is $\Cc^{\infty}$ on $\RR\times\RR$. \\

\noindent
Next, define
\[
X^r_t=|Y_t|,\quad U^r_t=V_t\,\text{sign}(Y)_t^+,\qquad \,t\geq 0,
\]
where $(\text{sign}(Y)_t^+,\,t\geq 0)$ is the c\`adl\`ag modification of $(\text{sign}(Y_t),\,t\geq 0)$.
Since, for any $n$, $\zeta_{n+1}=\inf\{t>\zeta_n,\,Y_t=0\}$ (with $\zeta_0=0$) is a predictable stopping time, $t\mapsto L_t$ never jumps a.s. at $\zeta_n$ (see e.g. Blumenthal \cite{Blumenthal-92}, Theorem $5.1$). It\^o's formula then yields that
\begin{align*}
X^r_t&=X_0+\int_0^t U^r_s\,ds,\\
U^r_t&=U_0+\int_0^t\text{sign}(Y)_s^- dL_s+\sum_{0<s\leq t}V_{s^-}\triangle \text{sign}(Y)_s^-\ind_{\{\triangle \text{sign}(Y)_s^-\neq 0\}} +\sum_{0<s\leq t}\triangle L_s\triangle \text{sign}(Y)_s^-\\
&=U_0+\int_0^t\text{sign}(Y)_s^- dL_s-2\sum_{0<s\leq t}U^r_{s^-}\ind_{\{\triangle X^r_s\neq 0\}}.
\end{align*}
Thanks to the symmetric property of $(L_t,\,t\geq 0)$ and the fact the $L$ and $\text{sign}(Y)$ a.s. do not jump at the same time, $(\int_0^t\text{sign}(Y)_s^- dL_s,\,t\geq 0)$ is also a symmetric $\alpha$-stable L\'evy process, and $((X^r_t,V^r_t),\,t\geq 0)$ is a weak solution to the Langevin model \eqref{eq:ConfinedModel} with purely elastic reflection. Therefore $\PP\left(X^r_t\in dx, U^r_t\in du\right)$ admits a smooth density function $\rho^r$ given by
\[
\PP\left(X^r_t\in dx, U^r_t\in du\right)=\left(\rho^f(t,x,u)+\rho^f(t,-x,-u)\right)\ind_{\{x\geq 0\}}\,dx\,du.
\]
Owing to the smoothness of $\mu_t^f$ and replicating the arguments of \cite[Theorem 2.3]{BoJa-11}, we deduce that the natural trace functions satisfying \eqref{eq:Variational1d} in the purely reflective case are given by
\[
\gamma^\pm(\rho^r)(t,0,u)=\left(\rho^f(t,0,u)+\rho^f(t,-0,-u)\right)\ind_{\{x\geq 0,\pm u< 0\}}.
\]
According to \eqref{eq:trace1}, this is enough to ensure that
\[
\PP\circ(\tau^r_1,U^r_{\tau_1^-})^{-1}=\PP\circ(\tau_1,U_0+L_{\tau_1^-})^{-1}
\]
admits a density with respect to $\left(\ind_{\{0\leq t\leq T\}}dt\right)\otimes\left( u\ind_{\{u\leq 0\}} du\right)$.
 \end{proof}
\begin{proof}[Proof of Theorem \ref{thm:tracethm}] For $(P_1)$, applying Lemma \ref{lem:2} and using \eqref{hyp:Maxwell} and the Markov property, we immediately deduce that for all $n\in\mathbb{N}$
\[
\PP(\tau_n\in dt,\,U_{\tau_n^-}\in du)
\]
admits a density with respect to the measure $u\ind_{\{u<0\}}du\otimes dt$, and that $(P_1)$ is satisfied.\\
For $(P_2)$, assuming that $0\leq p<1$, by the Markov property, we have
\begin{align*}
\EE\left[e^{-\tau_{n+1}}\right]=\EE\left[e^{-\tau_n}\EE_{(0,U_{\tau_{n}})}\left[e^{-\tau_1}\right]\right]%
&=\EE\left[e^{-\tau_n}\left((1-p)\EE_{(0,M_{n})}\left[e^{-\tau_1}\right]+p\EE_{(0,-cU_{\tau^-_{n}})}\left[e^{-\tau_1}\right]\right)\right]\\
&\leq \EE\left[e^{-\tau_n}\right]\EE\left[\left((1-p)\EE_{(0,M_{1})}\left[e^{-\tau_1}\right]+p\right)\right].
\end{align*}
Therefore, setting $\varrho:=\EE\left[(1-p)\EE_{(0,M_{1})}\left[e^{-\tau_1}\right]+p\right]$ which is strictly smaller than 1,
\[
\EE\left[e^{-\tau_{n+1}}\right]\leq \varrho^n\EE\left[e^{-\tau_1}\right].
\]
For any given $0<T<+\infty$, choosing $c_T>0$ such that $\ind_{\{r\leq T\}}\leq c_Te^{-r}$, it follows that
\[
\sum_{n\geq 1}\PP(\tau_n\leq T)\leq c_T\sum_{n\geq 1}\EE\left[e^{-\tau_n}\right]\leq c_T\EE\left[e^{-\tau_1}\right]\sum_{n\geq 0}\varrho^n<\infty.
\]
In the case where $p=1$ and $c>c_{\text{crit}}$, we first write, for $n\geq 2$ :
\begin{align*}
\sum_{n\geq 1}\PP\left(\tau_{n}\leq T\right)\leq \sum_{n\geq 1}\PP\left(\tau_{n+1}-\tau_n\leq T\right).
\end{align*}
Then, using the Markov's inequality with $\delta>0$ and the decomposition (\ref{eq:taun}),
\begin{align*}
\PP\left(\tau_{n+1}-\tau_n\leq T\right)&=\iint_{(0,\infty)\times\RR} \PP_{(x,u)}\left( \tau_{n+1}-\tau_n\leq T\right)\mu_0(dx,du)
\leq \iint_{(0,\infty)\times\RR} \EE_{(x,u)}\left[ \frac{T^\delta}{\left(\tau_{n+1}-\tau_n\right)^\delta}\right]\mu_0(dx,du)\\
&\leq T^\delta\EE\left[\xi_1^{-\delta}\right]\EE\left[\left(\prod_{j=1}^n\left|c\ell_i\right|^\alpha\right)^{-\delta}\right]\iint_{(0,\infty)\times\RR}\EE_{(x,u)}\left[ |U_{\tau_1^-}|^{-\delta}\right] \mu_0(dx,du)\\
&\leq T^\delta\EE\left[\xi_1^{-\delta}\right]
\left(\EE\left[\left|c\ell_1\right|^{-\alpha\delta}\right]\right)^n\iint_{(0,\infty)\times\RR}\EE_{(x,u)}\left[ |U_{\tau_1^-}|^{-\delta}\right] \mu_0(dx,du).
\end{align*}
According to Lemma \ref{lem:NegativeMoment}, taking $\delta <\nu$ immediately ensures that $\EE\left[\xi_1^{-\delta}\right]$ is finite. Next, since $f(\delta)= \EE\left[\left|c\ell_1\right|^{-\alpha\delta}\right]$ is such that $f(0)=1$ and
\[
f'(0)=\left(\EE\left[\exp\left(-\alpha\delta\ln\left|c\ell_1\right|\right)\right]\right)'_{|\delta=0}=-\alpha
\EE\left[\ln\left|c\ell_1\right|\right]<-\alpha\left(\ln(c_{\text{crit}})+\EE\left[\ln\left|\ell_1\right|\right]\right)=0,
\]
$f$ is decreasing near $0$. Hence, choosing $\delta>0$ small enough, $\varrho=\EE\left[\left|c\ell_1\right|^{-\alpha\delta}\right]<1$, and we get  $\PP\left(\tau_{n+1}-\tau_n\leq T\right)\leq C\varrho^n$ with $\varrho<1$. This enables to conclude $(P_2)$.
\end{proof}

\paragraph{Acknowledgements:}

We thank the referee for his/her careful reading and for pointing out a misstatement in the first version of this paper. The first author acknowledges the support of the Russian Academic Excellence Project '5-100'. Both authors acknowledge the partial support of the FONDECYT INICIACI\'ON Project N\textordmasculine  11130705, and the support of the N\'ucleo Milenio MESCD.
%\normalsize{

\end{document}